\documentclass[11pt,reqno,a4paper]{article} 
\usepackage{graphicx} 

\usepackage[top=1 in, bottom=1 in, left=1 in, right=1 in] {geometry}

\usepackage{amscd,amsmath,amstext,amsfonts,amsbsy,
amssymb,amsthm,mathrsfs,float,latexsym}
\usepackage{multicol,graphicx,array,multirow,color,lineno}
\usepackage{sidecap,url,fancybox,fancyhdr}
\usepackage[colorlinks=true,linkcolor=blue,citecolor=magenta]{hyperref}

\usepackage{comment}

\usepackage{subcaption}
\usepackage{enumitem}

\renewcommand{\epsilon}{\varepsilon}

\newtheorem{theorem}{{\bf Theorem}}[section]

\theoremstyle{plain}
\newtheorem{lemma}{Lemma}[section]
\newtheorem{proposition}{Proposition}[section]

\newtheorem{remark}{Remark}[section]

\usepackage{soul}  

\usepackage{pdfcomment} \usepackage[colorinlistoftodos]{todonotes}

\title{\Large \bf  Sharp estimates for Lyapunov exponents of Milstein approximation of stochastic differential systems}

\author{Vu Thi Hue\footnote{Email:  hue.vuthi@hust.edu.vn}}

\date{}

\begin{document}
\maketitle

\vspace*{-1cm}

\begin{center}
{\small
Faculty of Mathematics and Informatics, Hanoi University of Science and Technology, \\ 
No.1 Dai Co Viet Road, Hanoi, Vietnam  
}
\end{center}

\begin{abstract} The Milstein approximation with step size $\Delta t>0$ of the solution $(X, Y)$ to a two-by-two system of linear stochastic differential equations is considered. It is proved that when the solution of the underlying model is exponentially stable or exponentially blowing up at infinite time, these behaviours are preserved at the level of the Milstein approximate solution $\{(X_n, Y_n)\}$ in both the mean-square and almost-sure senses, provided sufficiently small step size $\Delta t$. This result is based on sharp estimates, from both above and below, of the discrete Lyapunov exponent. This type of sharp estimate for approximate solutions to stochastic differential equations seems to have been first studied in this work. 
In particular, the proposed method covers the setting for linear stochastic differential equations as well as the $\theta$-Milstein scheme's setting.

\medskip

\noindent \textbf{Keywords}: Milstein approximation, Sharp Lyapunov exponent, Exponential stability, Exponential blow-up.  
\end{abstract}

\section{Introduction}
Numerical approximations for stochastic differential equations are crucial and have been studied for many years. Among many constructed schemes, the Euler-Maruyama and Milstein schemes are among the most well-known, see e.g. \cite{Kloeden1992numerical,milstein2004stochastic}. In practice, it is frequently desirable that essential properties of the continuum solutions are also preserved at the discretised level. In the last two decades, the topic of this preservation has attracted a lot of attention, resulting in an extensive list of work including
\cite{appleby2008stabilization,bellman1985stability,braverman2020global,caraballo2003stochastic,hu2004discrete,kushner1968stability,mao1994stochastic,meerkov1982condition,saito1996stability, wu2010almost} and many references therein.


\medskip

This paper is devoted to studying the preservation of the exponential stability and exponential blow-up at infinite time
for the Milstein scheme of the following system of linear stochastic differential equations (SDEs), proposed in \cite{buckwar2010towards},  
\begin{equation}
    d \begin{pmatrix}
    X \\
    Y 
    \end{pmatrix} \!=\! 
    \begin{pmatrix}
    \lambda & 0 \\
    0 & \lambda
    \end{pmatrix} 
    \begin{pmatrix}
    X \\
    Y 
    \end{pmatrix} dt 
    + 
    \begin{pmatrix}
    \sigma & 0 \\
    0 & \sigma
    \end{pmatrix}\!\! 
    \begin{pmatrix}
    X \\
    Y 
    \end{pmatrix}  dB_1(t)  + 
    \varepsilon \begin{pmatrix}
    0 & -1 \\
    1 & 0
    \end{pmatrix} \!\!
    \begin{pmatrix}
    X \\
    Y 
    \end{pmatrix}  dB_2(t), 
\label{SDE:System}
\end{equation}
subjected to the initial condition
\begin{align}
    (X(0), Y(0))=(x_{0},y_{0}) \in \mathbb{R}^2 \setminus \{(0,0)\},
    \label{InitialData}
\end{align}
where  $\lambda \in \mathbb{R}$ is the drift coefficient, $\sigma, \varepsilon \in \mathbb{R}$ are the diffusion coefficients, and $B_1,B_2$ are standard scalar Brownian motions defined on a complete probability space $(\Omega
,\mathbb{P}, \mathcal F)$. The initial datum $(x_0,y_0)$ is assumed to be different from zero, which avoids the trivial solution. Our work focuses on large-time behaviours of the solution's module, for which, direct computations using the It\^o formula shows that  
\begin{align}
    d\log|Z(t)| = \left( \lambda - \frac{\sigma^2 - \varepsilon^2}{2} \right) dt + \sigma dB_1(t) , 
    \label{For:LogXY}
\end{align}
  where $Z:= [X \; Y]^{\top}$, and $|\cdot|$ is the Euclidean norm of $\mathbb R^2$. 

\medskip

The solution $Z$ to \eqref{SDE:System} is said to be \textit{mean-square exponentially stable}, respectively  \textit{almost-surely exponentially stable} with the exponent $\alpha>0$ if we have $$[ \mathbb{E}|Z(t)|^2]^{1/2} \le Ce^{-\alpha t}, \quad \text{respectively} \quad |Z(t)| \le Ce^{-\alpha t} \quad\text{almost-surely as } t \to \infty.$$ Conversely, if the previous sign ``$\le$" is replaced by ``$\ge$" and $-\alpha$ by $\beta>0$ then the solution is said to be \textit{mean-square exponentially blowing up} (at infinite time), or respectively, \textit{almost-surely exponentially blowing up} with the exponent $\beta$. We refer the reader to work  \cite{higham2007almost} for more details. 
It is well-known that these large-time behaviours are all encapsulated in the following limits, which can be directly obtained from \eqref{For:LogXY}, 
\begin{gather*}
      \limsup_{t\to\infty} \dfrac{1}{t} \log [ \mathbb{E}|Z(t)|^2]^{1/2} = \lambda + \frac{\varepsilon^2}{2} + \frac{\sigma^2}{2} , 
\end{gather*} 
and 
\begin{gather*}
   \limsup_{t\to\infty} \dfrac{1}{t} \log |Z(t)|  \, \stackrel{\text{a.s.}} {=} \, \lambda + \frac{\varepsilon^2}{2} - \frac{\sigma^2}{2} .
\end{gather*} 
Consequently, the solution $Z$ is mean-square or almost-surely exponentially stable with the exponent $\lambda+ \varepsilon^2/2 + \sigma^2/2$ or $\lambda+ \varepsilon^2/2 - \sigma^2/2$ if 
 $$(\lambda,\varepsilon, \sigma)\in \mathbf S^{\mathsf{ms}}:= \left\{(x,y,z) \in \mathbb{R}^3 \Big|\,x + \frac{y^2}{2} + \frac{z^2}{2} <0 \right\},$$ or respectively, if $$(\lambda, \varepsilon, \sigma)\in \mathbf S^{\mathsf{as}}:=\left\{(x,y,z) \in \mathbb{R}^3 \Big|\,x + \frac{y^2}{2} - \frac{z^2}{2} <0 \right\},$$
 and it is mean-square or almost-surely exponentially blowing up with the same exponents (as above) if $(\lambda, \varepsilon, \sigma)\in \mathbf B^{\mathsf{ms}} := \mathbb{R}^3 \setminus \overline {\mathbf S^{\mathsf{ms}} } $, or respectively, if $(\lambda, \varepsilon, \sigma)\in \mathbf B^{\mathsf{as}} := \mathbb{R}^3 \setminus \overline {\mathbf S^{\mathsf{as}} } $. Here, the $\overline{A}$ means the closure of $A$ in $\mathbb R^3$. 

 \medskip

The preservation of the above behaviours is expected to be retained for the Milstein approximation of \eqref{SDE:System} with a small enough step size $\Delta t$. Thanks to the expression \eqref{For:LogXY}, solution module $|Z|$ can be approximated as
\begin{align*}
    |Z(t+\Delta t)| & = |Z(t)| \exp \left[ \int_{t}^{t+\Delta t} \left( \lambda - \frac{\sigma^2 - \varepsilon^2}{2} \right) ds + \int_t^{t+\Delta t} \sigma dB_1(s)  \right] \\
    & \approx |Z(t)| \left [ 1 + \left( \left( \lambda + \frac{\varepsilon^2}{2} \right) - \frac{\sigma^2}{2}  \right)  \Delta t  +  \sigma   \Delta B_{1} +\frac{\sigma^2}{2}  \Delta B_{1}^2 \right ]. 
\end{align*}
Let $\{t_n\} :=\{ n \Delta t \} $,  for $n\in \mathbb{N}$, $n\ge 0$, be a discrete-time sequence, and  $\Delta B_{1,n}$ be the Brownian motion  increment $B_1(t_{n+1})-B_1(t_{n})$. The Milstein approximation of $|Z(t_n)|$ is given by   
\begin{align}
    |Z_n| 
    = |Z_{n-1}| \left [ 1 + \left( \left( \lambda + \frac{\varepsilon^2}{2} \right) - \frac{\sigma^2}{2}  \right)  \Delta t  +  \sigma   \Delta B_{1,n-1} +\frac{\sigma^2}{2}  \Delta B_{1,n-1}^2 \right ],
    \label{FinalLable}
\end{align}
for all $n\in \mathbb{N}$, $n\ge 1$, where $Z_0:=(x_0,y_0)$, see e.g.   \cite{milstein2004stochastic,Kloeden1992numerical} for general introduction. 

\medskip

Large-time behaviours of mean-square/almost-surely exponential stability of the zero equilibrium of the Euler-Maruyama scheme for \eqref{SDE:System} were studied 
in \cite[Section 4]{buckwar2010towards} via a lemma based on the martingale convergence theorem in \cite{Shiryaev1996probability}, where the instability or blow-up at infinite time had not been discussed. 
Then, in Berkolaiko-Buckwar-Kelly-Rodkina \cite{berkolaiko2012almost} (and its corrigendum \cite{berkolaiko2013corrigendum}),   the almost sure stability and blow-up at finite time of the zero equilibrium of the Euler-Maruyama scheme for \eqref{SDE:System} was obtained, where they showed that  
\begin{align}
    \lim_{n\to \infty} |Z_n|=0 \quad \text{if and only if} \quad  (\lambda,\varepsilon,\sigma) \in \mathbf{S}^{\mathsf{as}},
    \label{Pic:1}
\end{align}
and 
\begin{align}
    \lim_{n\to \infty} |Z_n|=\infty \quad \text{if and only if} \quad (\lambda,\varepsilon,\sigma) \in \mathbf{B}^{\mathsf{as}},
    \label{Pic:2}
\end{align}
see \cite[Theorem 3.2]{berkolaiko2012almost}. Unfortunately, based on the discrete It\^o formula, their method did not give an estimate for the discrete Lyapunov exponent $$ \limsup_{n\to\infty} \frac{1}{t_n} \log | Z_n |.$$ We notice that when $\varepsilon=0$, System \eqref{SDE:System} is reduced to an uncoupling system of linear stochastic differential equations, for which, standard analysis can be found in 
\cite{buckwar2010towards, kelly2013almost, mao2015almost}. 

\medskip 

Regarding mathematical techniques, estimating the Lyapunov exponents of numerical schemes \textit{from} \textit{above} has been studied quite early, such as close-to-sharp estimates in \cite[Section 3]{higham2007almost} and in many more papers \cite{buckwar2010towards,mao2015almost} with stability results on Euler-Maruyama and Milstein schemes as well as their relatives, where one of the essential steps is the utilisation of the fundamental inequality 
\begin{align}
    \log(1+x)\le x - \frac{x^2}{2} + \frac{x^3}{3}, \quad \text{for } x>-1 .
    \label{LogIneq}
\end{align}
However, constructing such estimates \textit{from below}, which is sharp in the sense that 
\begin{align}
     \limsup_{n\to\infty} \frac{1}{n\Delta t} \log | Z_n | \, \stackrel{\text{a.s.}} {\ge} \, \limsup_{t\to\infty} \dfrac{1}{t} \log |Z(t)| - C(\Delta t) 
    \label{SharpSense}
\end{align}
with $\lim_{\Delta t\to 0}C(\Delta t) = 0$, seems highly challenging. Up to the best of our knowledge, constructing lower estimates for the Lyapunov exponents of numerical schemes of stochastic differential equations has not been studied yet. 

\medskip 

In this paper, in both mean-square and almost-sure senses, the discrete Lyapunov exponent of \eqref{SDE:System} (i.e., the Lyapunov exponent of \eqref{FinalLable}) will be sharply estimated from both above and below, corresponding to the exponential stability and exponential blow-up of $|Z_n|$. Our main result is given as follows, where the abbreviation a.s. stands for the almost-sure property.

\begin{theorem} 
\label{Theo:Main}
Let $ \{Z_n\} $ be the solution to the Milstein scheme \eqref{FinalLable} with a given step size $0<\Delta t<1$. Then, there exist $0<\Delta t_1<1$ and $C_1,C_2>0$ such that
\begin{gather}
    \lambda + \frac{\varepsilon^2}{2} + \dfrac{\sigma^2}{2}  - C_1\Delta t  \le \limsup_{n\to\infty} \dfrac{1}{t_n} \log \big[ \mathbb{E}|Z_n|^2\big]^{1/2}  \le  \lambda + \frac{\varepsilon^2}{2} + \dfrac{\sigma^2}{2} + C_1\Delta t,  
    \label{Theo:Est:Main1}
    \\ 
 \lambda + \frac{\varepsilon^2}{2} - \dfrac{\sigma^2}{2} - C_2\Delta t^{1/2}   \stackrel{\text{a.s.}}{\le}  \limsup_{n\to\infty} \dfrac{1}{t_n} \log | Z_n | \stackrel{\text{a.s.}}{\le}   \lambda + \frac{\varepsilon^2}{2} - \dfrac{\sigma^2}{2} + C_2\Delta t^{1/2}, 
 \label{Theo:Est:Main2}
\end{gather}
for all $0<\Delta t< \Delta t_1$.  Consequently,   
there hold in the interval $(0,\Delta t_1)$ of step size that 
\begin{itemize}
    \item[a)] if 
 $(\lambda,\varepsilon,\sigma)\in \mathbf S^{\mathsf{ms}}$ or $(\lambda,\varepsilon,\sigma)\in \mathbf B^{\mathsf{ms}}$ then $| Z_n |$ is  mean-square exponentially stable, or respectively, mean-square exponentially blowing up with the exponent $\lambda + \varepsilon^2/2 + \sigma^2/2  $; 
    \item[b)] if 
 $(\lambda,\varepsilon,\sigma)\in \mathbf S^{\mathsf{as}}$ or $(\lambda,\varepsilon,\sigma)\in \mathbf B^{\mathsf{as}}$ then $| Z_n |$ is almost-sure exponentially stable, or respectively,  almost-sure exponentially blowing up with the exponent $\lambda + \varepsilon^2/2 - \sigma^2/2$.
\end{itemize} 
\end{theorem} 

By the above theorem, there exists a small time step $\Delta t_1>0$ such that the Milstein scheme \eqref{FinalLable} preserves the large-time behaviours of mean-square/almost-surely exponential stability of the zero equilibrium from the underlying model \eqref{SDE:System} whenever $0<\Delta t < \Delta t_1$. 

\medskip

Our novelty is estimating the discrete Lyapunov exponent from below that is sharp in the sense  \eqref{SharpSense}. We observe from the second order of the Milstein approximation that there exists a constant $\gamma_0$ such that $$\gamma_0 < \gamma_{\Delta t}:=  1 + \left( \left( \lambda + \frac{\varepsilon^2}{2} \right) - \frac{\sigma^2}{2}  \right)  \Delta t, \quad \text{for all } 0<\Delta t < \Delta t_1, $$ and  
\begin{align*}
    \mathbb{P} \left[   \sigma   \Delta B_{1,n-1} +\frac{\sigma^2}{2}  \Delta B_{1,n-1}^2 \ge -  \gamma_0 \right] =1,      
\end{align*} 
which avoids the singularity of the logarithmic function $x\mapsto \log(\gamma_{\Delta t} + x)$ at $x=-\gamma_{\Delta t}$, see Proposition \ref{Theo:2}. This suggests us to construct a lower bound for this function in the form 
  \begin{align*}
      \log(\gamma_{\Delta t}+x) \geq \log \gamma_{\Delta t} + \dfrac{1}{\gamma_{\Delta t}} x -\dfrac{1}{2\gamma_{\Delta t}^2} x^2 + \xi_{\gamma_{\Delta t}}(x), \quad \text{for all } x \ge -\gamma_0, 
  \end{align*}
with a suitable power function $\xi_{\gamma_{\Delta t}}$ of $x$, see Lemma \ref{Lem:BaIneqn}.

\begin{remark} Focused on the almost-sure sense, if the drift coefficient $\lambda$ is given, the stable and blowing-up regions of the noise parameters $\sigma,\varepsilon$ are shown in Figure \ref{fig4}, where the minus and plus signs correspond to the exponential stable and exponential blowing-up regions, respectively.

\medskip 

If $\lambda=0$, those regions are separated by the lines $\varepsilon=-\sigma$ and $\varepsilon=\sigma$, plotted in blue. If $\lambda>0$, they are separated by horizontal hyperbolae, which are wider as $\lambda$ increases, while it corresponds to vertical hyperbolae if $\lambda<0$.

\medskip

We notice that points lay on the curves or (blue) lines correspond to the zero value of the continuum Lyapunov exponent, i.e., the continuum solution is not stable or blowing up. However, estimate \eqref{Theo:Est:Main2} does not say that the discrete Lyapunov exponent is positive, negative, or zero. Thus, a conclusion about the large-time behaviour of the discrete solution is not apparent. 
 
\begin{figure}[H]
\centering
\includegraphics[width=8cm, height=8cm]{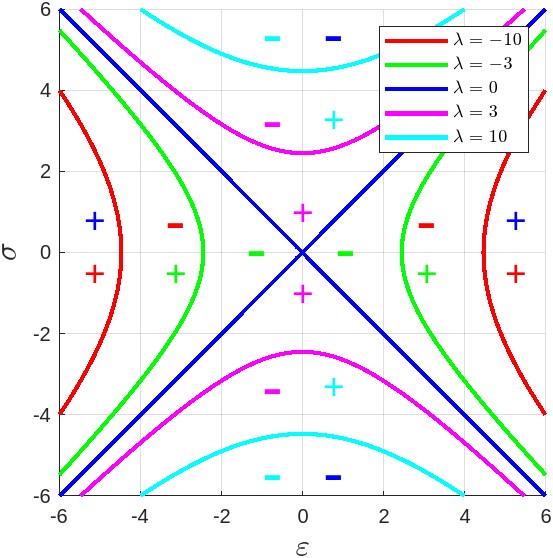}
    \caption{Stable and blowing-up regions of noise parameters $\sigma,\varepsilon$ given drift coefficient $\lambda$.}
\label{fig4}
\end{figure}    
\end{remark}
 
This paper is organised as follows. Proof of Theorem \ref{Theo:Main} is obtained by combining Propositions \ref{Theo:1}-\ref{Theo:2}, which will be presented in Section 2. In Section 3, we will point out that the proposed method also covers the setting for linear stochastic differential equations
as well as the $\theta$-Milstein scheme’s setting. Finally, the last section demonstrates some numerical simulations to illustrate the main result in the almost-sure sense.

\section{Exponential stability versus exponential blow-up}

\subsection{Mean-square exponential stability}
 
In this part, we show the exponential stability of the zero equilibrium of \eqref{FinalLable} versus the mean-square exponential blow-up (at infinite time) in the mean-square sense, which will be presented in Proposition \ref{Theo:1} below. For the sake of convenience, we denote 
\begin{align}
    \gamma_{\Delta t} &:= 1 + \left( \lambda + \frac{\varepsilon^2}{2} - \frac{\sigma^2}{2}  \right)  \Delta t, 
    \label{gammaDelta}\\
    \mu &:= \lambda ^2 + \lambda \varepsilon^2+ \frac{\varepsilon^4}{4} +\frac{\sigma^4}{2}.
    \label{mu}
\end{align}
Since the numerical step size $\Delta t$ will be chosen sufficiently small, we can assume without loss of generality that $\gamma_{\Delta t}$ is strictly positive throughout this paper.

\begin{proposition}[Exponential stability versus exponential blow-up in the mean-square sense]
\label{Theo:1}
Let $\{Z_n\}$ be the solution to the Milstein scheme \eqref{FinalLable} with a step size $0<\Delta t<1$. Then, it holds   
\begin{align}
\limsup_{n\to\infty} \frac{1}{t_n} \log \mathbb{E}(Z_n^2) 
  = \,  2\left( \lambda + \frac{\varepsilon^2}{2} + \frac{\sigma^2}{2} \right)   + \mathbf{R}_{\lambda,\varepsilon,\sigma}^{\mathsf{ms}}(\Delta t),
  \label{Limit:1}
\end{align}  
where  the remainder   $\mathbf{R}_{\lambda,\varepsilon,\sigma}^{\mathsf{ms}}(\Delta t)$ is formulated as
\begin{align}
\begin{aligned}
\mathbf{R}_{\lambda,\varepsilon,\sigma}^{\mathsf{ms}}(\Delta t) :=  \mu \Delta t   + \sum_{m=2}^{\infty} \frac{(-1)^{m-1}}{m} \left[ 2 \left(\lambda + \frac{\varepsilon^2}{2} + \frac{\sigma^2}{2} \right) +  \mu \Delta t \right]^m \Delta t^{m-1}.
\label{Rem2}
\end{aligned}
\end{align}
Moreover, there exists $\Delta t_\dagger>0$ sufficiently small such that, for all $0<\Delta t<\Delta t_\dagger$,
 \begin{align*} 
\left|\mathbf{R}_{\lambda,\varepsilon,\sigma}^{\mathsf{ms}}(\Delta t) \right|
 \le  \left( \mu  +  \frac{ \big[ 2|\lambda +  \frac{\varepsilon^2}{2} + \frac{\sigma^2}{2} |+\mu \Delta t \big]^2}{ 1 - \big[ 2|\lambda + \frac{\varepsilon^2}{2} + \frac{\sigma^2}{2} |+\mu \Delta t \big] \Delta t } \right) \Delta t .
\end{align*}
\end{proposition}

\begin{proof} By taking the square of the Milstein scheme's solution \eqref{FinalLable}, we get 
    \begin{align} 
         Z_n^2 = ( x_0^2 + y_0^2 )\prod_{k=0}^{n-1} \left[ \gamma_{\Delta t}^2  + 2 \gamma_{\Delta t}   \left (\sigma   \Delta B_{1,k} +\frac{\sigma^2}{2}  \Delta B_{1,k}^2 \right )  + \left ( \sigma   \Delta B_{1,k} +\frac{\sigma^2}{2}  \Delta B_{1,k}^2\right )^2 \right] , 
    \label{Theo:1:P1}
\end{align} 
where $\gamma_{\Delta t}$ was defined at \eqref{gammaDelta}. We recall the basic expectations  $$\mathbb{E}[(\Delta B_{1,k})^{2i+1}]=0 \quad \text{and} \quad \mathbb{E}[(\Delta B_{1,k})^{2i}]=(2i-1)!! (\Delta t)^i, \quad i\in \mathbb{N},i\ge 1,$$ 
which directly show that 
$$\mathbb{E} \left (\sigma   \Delta B_{1,k} + \frac{\sigma^2}{2} \Delta B_{1,k}^2 \right )= \frac{\sigma^2}{2}\Delta t$$ and $$\mathbb{E}\left[\left(\sigma   \Delta B_{1,k} + \frac{\sigma^2}{2}  \Delta B_{1,k}^2\right)^2\right]=\sigma^2\Delta t + \frac{3\sigma^4}{4} \Delta t^2.$$  
Hence, the expectation of each  product's factor in \eqref{Theo:1:P1} is given by 
\begin{align*}
    & \mathbb{E} \left[ \gamma_{\Delta t}^2  + 2 \gamma_{\Delta t} \left ( \sigma   \Delta B_{1,k} +\frac{\sigma^2}{2}  \Delta B_{1,k}^2 \right)  +  \left (\sigma   \Delta B_{1,k} +\frac{\sigma^2}{2}  \Delta B_{1,k}^2\right)^2 \right]\\  & =  \gamma_{\Delta t}^2 + \gamma_{\Delta t} \sigma^2\Delta t  + \sigma^2 \Delta t+\frac{3\sigma^4}{4}  \Delta t^2, \quad \text{for all }  k=0,1,\dots,n-1,  
 \end{align*}
and all product factors' expectations in \eqref{Theo:1:P1} are the same. 
Therefore,  
\begin{align}
     \begin{aligned}
         \mathbb{E}(Z_n^2)  
     & =  (x_0^2 +y_0^2)  \left[ 1+ (2\lambda +\varepsilon^2 + \sigma^2 ) \Delta t + \mu \Delta t^2\right]^n, 
     \end{aligned} 
    \label{Form:Induction}
\end{align}
where $\mu$ is defined at \eqref{mu}. 

\medskip

We note that the base $1+ (2\lambda + \varepsilon^2 + \sigma^2 ) \Delta t + 
\mu \Delta t^2$ is nonnegative due to taking the square of $| Z_n |$. To estimate the mean-square Lyapunov exponent, we restrict $\Delta t$ to an interval $(0,\widetilde{\Delta t})$ for a small $\widetilde{\Delta t}>0$ such that this base is strictly positive. Then, taking the logarithms of two sides of the latter expression and using the Taylor expansion of the logarithmic function $x\mapsto \log(1+x)$ for $x>-1$, we obtain
\begin{align*}
    \log \mathbb{E}(Z_n^2)  =  \log  (x_0^2+y_0^2)  + n \sum_{m=1}^{\infty} \frac{(-1)^{m-1}}{m} \left[ 2 \left(\lambda +\frac{\varepsilon^2}{2} + \frac{\sigma^2}{2} \right) \Delta t   +  \mu \Delta t^2 \right]^m.
\end{align*}
This subsequently gives
\begin{align*}
\limsup_{n\to\infty} \frac{1}{t_n} \log \mathbb{E}(Z_n^2) 
& =  \dfrac{1}{\Delta t} \sum_{m=1}^{\infty} \frac{(-1)^{m-1}}{m} \left[ 2 \left(\lambda + \frac{\varepsilon^2}{2} +\frac{\sigma^2}{2} \right) \Delta t +
 \mu \Delta t^2 \right]^m\\
&=  \sum_{m=1}^{\infty} \frac{(-1)^{m-1}}{m} \left[ 2 \left(\lambda + \frac{\varepsilon^2}{2} + \frac{\sigma^2}{2} \right) + \mu \Delta t \right]^m \Delta t^{m-1} ,
\end{align*}
and then,
\begin{align*}
\limsup_{n\to\infty} \frac{1}{t_n} \log \mathbb{E}(Z_n^2) = 2 \left( \lambda + \frac{\varepsilon^2}{2} + \frac{\sigma^2}{2} \right) + \mathbf{R}_{\lambda,\varepsilon,\sigma}^{\mathsf{ms}}(\Delta t) 
\end{align*}
in which the remainder $\mathbf{R}_{\lambda,\varepsilon,\sigma}^{\mathsf{ms}}(\Delta t)$ was defined at \eqref{Rem2}. Now, we only have to estimate the remainder to complete the proof. Using the triangle inequality, 
\begin{align*}
\left|\mathbf{R}_{\lambda,\varepsilon,\sigma}^{\mathsf{ms}}(\Delta t) \right|  
 &\le \mu \Delta t  + \left( 2 \left| \lambda + \frac{\varepsilon^2}{2} + \frac{\sigma^2}{2} \right|+ \mu \Delta t \right)   \sum_{m=2}^{\infty} \left[ \left( 2 \left| \lambda + \frac{\varepsilon^2}{2} + \frac{\sigma^2}{2} \right|+ \mu \Delta t \right) \Delta t \right]^{m-1},
\end{align*}
and thus,
\begin{align*}
\left|\mathbf{R}_{\lambda,\varepsilon,\sigma}^{\mathsf{ms}}(\Delta t) \right| \le   \mu \Delta t   + \frac{ \big[ 2|\lambda + \frac{\varepsilon^2}{2} + \frac{\sigma^2}{2} |+\mu \Delta t \big]^2}{ 1 - \big[ 2|\lambda + \frac{\varepsilon^2}{2} + \frac{\sigma^2}{2} |+\mu \Delta t \big] \Delta t }  \Delta t ,
\end{align*}
where we chose $\Delta t_\dagger$ small enough such that $0<\Delta t_\dagger<\widetilde{\Delta t}$ and  $$\left( 2\left|\lambda + \frac{\varepsilon^2}{2}+ \frac{\sigma^2}{2} \right|+\mu \Delta t \right) \Delta t<1$$ for all $0<\Delta t <\Delta t_\dagger$ to have the convergence of the latter infinite series. 
\end{proof}

\begin{remark} In \cite[Section 4]{higham2000mean}, the stability versus the blow-up in the mean-square sense was claimed by respectively arguing from  \eqref{Form:Induction} in two cases that the value of $$1+ (2\lambda +\varepsilon^2 + \sigma^2 ) \Delta t + \mu \Delta t^2$$ is strictly less or greater than $1$, where, however, an exponential rate of the stability or the blow-up had not been established.
\end{remark}

\subsection{Almost-sure exponential stability}

This part is to study the almost-sure exponential stability of the zero equilibrium of \eqref{SDE:System} versus its almost-sure exponential blow-up. We first need the lower and upper bounds for a logarithmic function in the following lemma, where the proof of this lemma will be put in the appendix.  

\begin{lemma}
\label{Lem:BaIneqn}
Given $\gamma>0$. There hold that  
\begin{align}
    \log(\gamma+x) \leq \log  \gamma  + \frac{1}{\gamma} x   -\frac{1}{2\gamma^2} x^2 +\frac{1}{3\gamma^3} x^3 \quad \text{for all } x > - \gamma, 
    \label{BaIneqn1}
\end{align}
and  
    \begin{align}
         \log(\gamma+x) \geq \log \gamma + \dfrac{1}{\gamma} x -\dfrac{1}{2\gamma^2} x^2 + \xi_\gamma(x)  \quad \text{for all } x > - \frac{2}{3} \gamma ,
    \label{BaIneqn2}
    \end{align}
where 
\begin{align*}
    \xi_\gamma(x):=  
    \left\{ \begin{array}{llrll}
    - \dfrac{x^4}{4\gamma^4} & \text{if}& x \ge 0,
    \vspace{0.15cm}\\
    \dfrac{9x^3}{\gamma^3}  & \text{if}&   - \dfrac{2}{3} \gamma < x < 0.  
    \end{array} \right.
\end{align*}
\end{lemma}

The next lemma is to estimate the expectation of $\xi_{\gamma_{\Delta t}}(\sigma   \Delta B_{1,k} + (\sigma^2/2) \Delta B_{1,k}^2)$, which we will exploit in Proposition \ref{Theo:2}. 

\begin{lemma}
\label{Lem:BaExp} Let $\xi_\gamma$ be the function defined by Lemma \ref{Lem:BaIneqn}. Then, it holds that 
    \begin{align}
        \mathbb{E} \left[ \xi_{\gamma_{\Delta t}}\left( \sigma   \Delta B_{1,k} +\frac{\sigma^2}{2}  \Delta B_{1,k}^2 \right) \right] \ge -  C \left( \frac{|\sigma \sqrt{\Delta t}|^3}{\gamma_{\Delta t}^3} +  \frac{|\sigma \sqrt{\Delta t}|^4}{\gamma_{\Delta t}^4} \right) ,  
        \label{Lem:BaExp:S}
    \end{align}
    for all $k \in \mathbb{N}$, $k\ge 0$.
\end{lemma}

\begin{proof} By the definition of the function $\xi_{\gamma}$ in Lemma \ref{Lem:BaIneqn}, we have 
    \begin{align*}
     \xi_{\gamma_{\Delta t}}\left( \sigma   \Delta B_{1,k} +\frac{\sigma^2}{2}  \Delta B_{1,k}^2 \right)    = \left\{ \begin{array}{llrll}
     - \dfrac{1}{4\gamma_{\Delta t}^4} \Big(\sigma   \Delta B_{1,k} +\dfrac{\sigma^2}{2}  \Delta B_{1,k}^2\Big)^4 & \text{if}& \omega \in \Omega_1,
    \vspace{0.15cm}\\
    \dfrac{9}{\gamma_{\Delta t}^3} \Big(\sigma   \Delta B_{1,k} +\dfrac{\sigma^2}{2}  \Delta B_{1,k}^2\Big)^3  & \text{if}& \omega \in \Omega_2,  
    \end{array} \right.  
\end{align*}
where $\Omega_1$ and $\Omega_2$ are subsets of the sample space $\Omega$, defined by
\begin{gather*}
    \Omega_1 := \left\{ \omega \in \Omega \, \bigg| \, \sigma   \Delta B_{1,k} +\frac{\sigma^2}{2}  \Delta B_{1,k}^2 \ge 0  \right\}, \\
    \Omega_2 := \left\{ \omega \in \Omega \, \bigg| \,- \dfrac{2}{3} \gamma_{\Delta t} < \sigma   \Delta B_{1,k} +\frac{\sigma^2}{2}  \Delta B_{1,k}^2 < 0  \right\}.
\end{gather*}
Let us denote $\zeta_k := \frac{\Delta B_{1,k}}{\sqrt{\Delta t}} $ for each $k\ge 0$. Then, $\zeta_k \sim N(0,1)$ (i.e., a standard normal distribution), and therefore, 
\begin{align*}
    \text{LHS } \eqref{Lem:BaExp:S}  = &\frac{9}{\gamma_{\Delta t}^3\sqrt{2\pi}} \int_{ \left\{   -\frac{2}{3} \gamma_{\Delta t} < (\sigma \sqrt{\Delta t}) y +\frac{(\sigma \sqrt{\Delta t})^2}{2}  y^2 < 0  \right\} } \left( (\sigma \sqrt{\Delta t})  y +\frac{(\sigma \sqrt{\Delta t})^2}{2}  y^2 \right)^3 e^{-\frac{y^2}{2}} dy \\ 
    & - \frac{1}{4\gamma_{\Delta t}^4\sqrt{2\pi}} \int_{\left\{    (\sigma \sqrt{\Delta t})  y +\frac{(\sigma \sqrt{\Delta t})^2}{2}  y^2 \ge 0  \right\}} \left( (\sigma \sqrt{\Delta t})  y +\frac{(\sigma \sqrt{\Delta t})^2}{2}  y^2 \right)^4 e^{-\frac{y^2}{2}} dy .
\end{align*}
By taking the absolute value, this shows 
\begin{align*}
    \text{LHS } \eqref{Lem:BaExp:S}    \ge & -  \frac{9}{\gamma_{\Delta t}^3\sqrt{2\pi}} \int_{ \left\{-   \frac{2}{3} \gamma_{\Delta t} < (\sigma \sqrt{\Delta t}) y +\frac{(\sigma \sqrt{\Delta t})^2}{2}  y^2 < 0  \right\} } \left( |\sigma \sqrt{\Delta t}  \, y | +\frac{(\sigma \sqrt{\Delta t})^2}{2}  y^2 \right)^3 e^{-\frac{y^2}{2}} dy \\ 
    & - \frac{1}{4\gamma_{\Delta t}^4\sqrt{2\pi}} \int_{\left\{    (\sigma \sqrt{\Delta t})  y +\frac{(\sigma \sqrt{\Delta t})^2}{2}  y^2 \ge 0  \right\}} \left( |\sigma \sqrt{\Delta t}  \, y| +\frac{(\sigma \sqrt{\Delta t})^2}{2}  y^2 \right)^4 e^{-\frac{y^2}{2}} dy \\
    \ge & -  \frac{9C_\sigma|\sigma \sqrt{\Delta t}|^3}{\gamma_{\Delta t}^3\sqrt{2\pi}} \int_{-\infty}^{\infty} (|wy|^3 + y^6) e^{-\frac{y^2}{2}} dy  - \frac{C_\sigma|\sigma \sqrt{\Delta t}|^4}{4\gamma_{\Delta t}^4\sqrt{2\pi}} \int_{-\infty}^{\infty} (y^4+y^8) e^{-\frac{y^2}{2}} dy \\
    \ge & -  \frac{9C_\sigma|\sigma \sqrt{\Delta t}|^3}{\gamma_{\Delta t}^3\sqrt{2\pi}} \int_{-\infty}^{\infty} \left(1 + \frac{5}{4}
 y^6 \right) e^{-\frac{y^2}{2}} dy  - \frac{C_\sigma|\sigma \sqrt{\Delta t}|^4}{4\gamma_{\Delta t}^4\sqrt{2\pi}} \int_{-\infty}^{\infty} (y^4+y^8) e^{-\frac{y^2}{2}} dy.
\end{align*}
Consequently, we get
\begin{align*}
    \text{LHS } \eqref{Lem:BaExp:S}    & \ge  -  \frac{9C_\sigma|\sigma \sqrt{\Delta t}|^3}{\gamma_{\Delta t}^3} \left( 1 + \frac{5}{6} \mathbb{E}(\zeta_k^6) \right)   - \frac{C_\sigma|\sigma \sqrt{\Delta t}|^4}{4\gamma_{\Delta t}^4}  \left( \mathbb{E}(\zeta_k^4) + \mathbb{E}(\zeta_k^8) \right) \\
    & = -  \frac{9C_\sigma|\sigma \sqrt{\Delta t}|^3}{\gamma_{\Delta t}^3} \times \dfrac{81}{6}  - \frac{C_\sigma|\sigma \sqrt{\Delta t}|^4}{4\gamma_{\Delta t}^4} \times 108 , 
\end{align*}
where we recall that $E(\zeta_k^{2m-1})=0$ and $E(\zeta_k^{2m})=(2m-1)!!$ for all $k\ge 0$ and $m\in \mathbb{N}$, $m\ge 1$. The latter estimate directly gives the desired one. 
\end{proof}

\begin{proposition}[Exponential stability versus exponential blow-up in the almost-sure sense]
\label{Theo:2} 
Let $\{Z_n\}$ be the solution to the Milstein scheme \eqref{FinalLable} with a step size $0<\Delta t<1$. Then, there exist $\Delta t_\ddagger>0$ and $0<C_-<C_+$ such that  
\begin{align}
\limsup_{n\to \infty} \frac{1}{t_n} \log |Z_n| 
 & \, \stackrel{\text{a.s.}} {\ge} \, \frac{1}{\Delta t} \log \left( 1 + \left( \lambda +\frac{\varepsilon^2}{2} - \frac{\sigma^2}{2}  \right)  \Delta t \right) - C_- \Delta t^{1/2},  
  \label{Limit:2a} \\
  \limsup_{n\to \infty} \frac{1}{t_n} \log |Z_n| 
 &\, \stackrel{\text{a.s.}} {\le} \, \frac{1}{\Delta t} \log \left( 1 + \left( \lambda +\frac{\varepsilon^2}{2} - \frac{\sigma^2}{2}  \right)  \Delta t \right) + C_+  \Delta t^{1/2},  
  \label{Limit:2b}
\end{align}  
for all $0<\Delta t<\Delta t_\ddagger$. 
\end{proposition}

\begin{proof} With the notation $\gamma_{\Delta t}$ defined at \eqref{gammaDelta}, we follow from the solution formula \eqref{FinalLable} that 
\begin{align*}
    |Z_n| = \sqrt{x_0^2+y_0^2} \; \prod_{k=0}^{n-1} \left( \gamma_{\Delta t}   +   \sigma   \Delta B_{1,k} +\frac{\sigma^2}{2}  \Delta B_{1,k}^2  \right) .     
\end{align*}    
Subsequently, after taking the absolute value of $Z_n$ and the logarithm of $|Z_n|$,   we get
\begin{align}
 \log |Z_n| &= \log \sqrt{x_0^2+y_0^2} +   \sum_{k=0}^{n-1} \log  
 \left |  \gamma_{\Delta t}  +   \sigma   \Delta B_{1,k} +\frac{\sigma^2}{2}  \Delta B_{1,k}^2  \right |. 
 \label{LogForm}
\end{align}

In the following, we split this proof into two parts. The first one is to prove the upper boundedness \eqref{Limit:2a}, while the second one is to prove the lower boundedness \eqref{Limit:2b}. 

\medskip

\noindent \underline{\it Lower boundedness:} We first note that it is possible to choose $\Delta t_0$ small enough so that $\gamma_{\Delta t} > 3/4$ for all $0<\Delta t <\Delta t_0$. Therefore, by observing  
\begin{align}
     \sigma   \Delta B_{1,k} +\frac{\sigma^2}{2}  \Delta B_{1,k}^2 &  =    \frac{( \sigma   \Delta B_{1,k} +1)^2}{2}  -  \frac{1}{2} , 
    \label{MIA1}
\end{align}
we see in this interval of $\Delta t$ that 
\begin{align}
      \sigma   \Delta B_{1,k} +\frac{\sigma^2}{2}  \Delta B_{1,k}^2 > - \frac{1}{2} > - \frac{2}{3} \gamma_{\Delta t} \quad \text{for all } \omega \in \Omega .  
      \label{MIA2}
\end{align}
To prove the lower boundedness \eqref{Limit:2a}, we will employ the inequality \eqref{BaIneqn2}.
Indeed, thanks to \eqref{MIA1}, we have $\gamma_{\Delta t}  +   \sigma   \Delta B_{1,k} + (\sigma^2/2) \Delta B_{1,k}^2 >0 $ for $0<\Delta t < \Delta t_0$. 
Therefore, we imply from the inequalities \eqref{BaIneqn2} and  \eqref{MIA2} that 
\begin{align*}
    & \log  
 \left |  \gamma_{\Delta t}  +   \sigma   \Delta B_{1,k} +\frac{\sigma^2}{2}  \Delta B_{1,k}^2  \right |  = \log  
 \left [  \gamma_{\Delta t}  +  \left( \sigma   \Delta B_{1,k} +\frac{\sigma^2}{2}  \Delta B_{1,k}^2 \right)  \right ] \\
 & \ge \log \gamma_{\Delta t} + \frac{1}{\gamma_{\Delta t}} \left( \sigma   \Delta B_{1,k} +\frac{\sigma^2}{2}  \Delta B_{1,k}^2 \right) -  \frac{1}{2\gamma_{\Delta t}^2} \left( \sigma   \Delta B_{1,k} +\frac{\sigma^2}{2}  \Delta B_{1,k}^2 \right)^2 \\
 & + \xi_{\gamma_{\Delta t}}\left( \sigma   \Delta B_{1,k} +\frac{\sigma^2}{2}  \Delta B_{1,k}^2 \right) , 
\end{align*}
where the last term is defined by Lemma \ref{Lem:BaExp}. 
By the Kolmogorov strong law of large numbers, for all $k=0,1,\dots,n-1$, we have 
\begin{align}
      \begin{aligned}
      \limsup_{n\to \infty} \frac{1}{t_n}  \sum_{k=0}^{n-1} \frac{  \sigma   \Delta B_{1,k} +\frac{\sigma^2}{2}  \Delta B_{1,k}^2  }{ \gamma_{\Delta t} }  \stackrel{\text{a.s.}}{=}  \frac{1}{\Delta t}\mathbb{E} 
      \left( \frac{  \sigma   \Delta B_{1,k} +\frac{\sigma^2}{2}  \Delta B_{1,k}^2  }{\gamma_{\Delta t}} \right)
      =    \frac{\sigma^2}{2\gamma_{\Delta t}}  , 
      \end{aligned}
      \label{Stronglaw:Ap1}  
\end{align}
and
\begin{align}
   \begin{aligned}
        \limsup_{n\to \infty} \frac{1}{t_n} \sum_{k=0}^{n-1} \left( \frac{ \sigma   \Delta B_{1,k} +\frac{\sigma^2}{2}  \Delta B_{1,k}^2 }{\gamma_{\Delta t}} \right )^2  \stackrel{\text{a.s.}}{=} \frac{1}{\Delta t}\mathbb{E}\left( \frac{ \sigma   \Delta B_{1,k} +\frac{\sigma^2}{2}  \Delta B_{1,k}^2 }{\gamma_{\Delta t}} \right )^2  = \frac{4\sigma^2 +  3\sigma^4 \Delta t }{4\gamma_{\Delta t}^2} .
   \end{aligned} \label{Stronglaw:Ap2}
\end{align} 
While, Lemma \ref{Lem:BaExp} gives 
\begin{align*}
        \mathbb{E} \left[ \xi_{\gamma_{\Delta t}}\left( \sigma   \Delta B_{1,k} +\frac{\sigma^2}{2}  \Delta B_{1,k}^2 \right) \right] \ge -  C \left( \frac{|\sigma \sqrt{\Delta t}|^3}{\gamma_{\Delta t}^3} +  \frac{|\sigma \sqrt{\Delta t}|^4}{\gamma_{\Delta t}^4} \right) ,  
    \end{align*}
which allows us to apply the Kolmogorov strong law of large numbers to get that 
\begin{align*}
    & \limsup_{n\to \infty} \frac{1}{t_n} \sum_{k=0}^{n-1} \xi_{\gamma_{\Delta t}}\left( \sigma   \Delta B_{1,k} +\frac{\sigma^2}{2}  \Delta B_{1,k}^2 \right) \\
    &  = \frac{1}{\Delta t} \mathbb{E} \left[ \xi_{\gamma_{\Delta t}}\left( \sigma   \Delta B_{1,k} +\frac{\sigma^2}{2}  \Delta B_{1,k}^2 \right) \right] \ge -  C \left( \frac{|\sigma|^3 \sqrt{\Delta t}}{\gamma_{\Delta t}^3} +  \frac{\sigma^4 \Delta t}{\gamma_{\Delta t}^4} \right) .
\end{align*}
Taking the above arguments, it follows from \eqref{LogForm} that 
\begin{align*}
\limsup_{n\to \infty} \frac{1}{t_n} \log |Z_n| 
 & =    \limsup_{n\to \infty} \frac{1}{t_n}    \sum_{k=0}^{n-1} \log  
 \left |  \gamma_{\Delta t}  +   \sigma   \Delta B_{1,k} +\frac{\sigma^2}{2}  \Delta B_{1,k}^2  \right | \\
 & \ge \frac{1}{\Delta t} \log \gamma_{\Delta t} + \frac{\sigma^2}{2\gamma_{\Delta t}^2} - \frac{4\sigma^2 +  3\sigma^4 \Delta t }{8\gamma_{\Delta t}^4} -  C \left( \frac{|\sigma|^3 \sqrt{\Delta t}}{\gamma_{\Delta t}^3} +  \frac{\sigma^4 \Delta t}{\gamma_{\Delta t}^4} \right) \\
 & = \frac{1}{\Delta t} \log \left( 1 + \left( \lambda +\frac{\varepsilon^2}{2} - \frac{\sigma^2}{2}  \right)  \Delta t \right) + O(\sqrt{\Delta t}), 
\end{align*}
where we notice that
\begin{align*}
    \frac{\sigma^2}{2\gamma_{\Delta t}^2} - \frac{4\sigma^2 }{8\gamma_{\Delta t}^4} = \frac{\sigma^2(\gamma_{\Delta t}^2-1)}{2\gamma_{\Delta t}^4} = O (\Delta t). 
\end{align*}

\noindent \underline{\it Upper boundedness:} To prove the upper boundedness \eqref{Limit:2b}, for each $\gamma>0$, we use the inequality \eqref{BaIneqn1}. In fact,
by restricting the step size $\Delta t$ to the interval $(0,\Delta t_0)$, this inequality \eqref{BaIneqn1}  shows   
\begin{align*}
    & \log  
 \left |  \gamma_{\Delta t}  +   \sigma   \Delta B_{1,k} +\frac{\sigma^2}{2}  \Delta B_{1,k}^2  \right |  = \log  
 \left [  \gamma_{\Delta t}  +  \left( \sigma   \Delta B_{1,k} +\frac{\sigma^2}{2}  \Delta B_{1,k}^2 \right)  \right ] \\
 & \le \log \gamma_{\Delta t} + \frac{1}{\gamma_{\Delta t}} \left( \sigma   \Delta B_{1,k} +\frac{\sigma^2}{2}  \Delta B_{1,k}^2 \right) -  \frac{1}{2\gamma_{\Delta t}^2} \left( \sigma   \Delta B_{1,k} +\frac{\sigma^2}{2}  \Delta B_{1,k}^2 \right)^2 \\
 & + \frac{1}{3\gamma_{\Delta t}^3} \left( \sigma   \Delta B_{1,k} +\frac{\sigma^2}{2}  \Delta B_{1,k}^2 \right)^3 , 
\end{align*}
where the limit superior related to the last term can be calculated via the Kolmogorov strong law of large numbers, similarly to \eqref{Stronglaw:Ap1}-\eqref{Stronglaw:Ap2}. The other computations to obtain \eqref{Limit:2b} is performed in the same way as the proof of the lower bound. 
\end{proof}


\begin{remark}
\label{Rem:Fundamental}
It is well-known that utilising the  fundamental inequality
$$\log(1+x)\le x - \frac{x^2}{2} + \frac{x^3}{3}, \quad \text{for } x>-1,$$
has been permanently used to estimate the discrete Lyapunov exponent (in the almost-sure sense), see e.g. \cite[Theorem 4.2]{higham2007almost} for more detail. If we use this inequality to estimate the logarithmic factor in \eqref{LogForm}, then, by writing the logarithmic factor in \eqref{LogForm} as 
\begin{align}
     \begin{aligned}
     \log  
 \left |  \gamma_{\Delta t}  +   \sigma   \Delta B_{1,k} +\frac{\sigma^2}{2}  \Delta B_{1,k}^2  \right |   =  \log  
 \left | 1 +x_*  \right | ,
     \end{aligned}
     \label{Rem:UseofTE:S}
\end{align} 
it holds   
\begin{align*}
\log  
 \left |  \gamma_{\Delta t}  +   \sigma   \Delta B_{1,k} +\frac{\sigma^2}{2}  \Delta B_{1,k}^2  \right | \le x_* - \frac{1}{2} x_*^2  + \frac{1}{3} x_*^3  ,
\end{align*}
where 
\begin{align}
    x_* = \left( \lambda + \frac{\varepsilon^2}{2} - \frac{\sigma^2}{2}  \right)  \Delta t  +   \sigma   \Delta B_{1,k} +\frac{\sigma^2}{2}  \Delta B_{1,k}^2.
\end{align}
Here, $\Delta t$ has been chosen such that $\gamma_{\Delta t}  +   \sigma   \Delta B_{1,k} + (\sigma^2/2)  \Delta B_{1,k}^2$ is almost-surely positive.  
\end{remark}

\begin{remark} It had been showed in \cite{berkolaiko2012almost} that the discrete It\^o formula gave an applicable approach to obtaining a picture of the equivalence between almost-sure stability and almost-sure blow-up \eqref{Pic:1}-\eqref{Pic:2}. However, these limits have not provided a mathematical tool to estimate the discrete Lyapunov exponent, or in other words, to obtain the equivalence between almost-surely \textit{exponential}  stability and almost-surely \textit{exponential} blow-up.      
\end{remark}

\section{Linear SDEs and theta-Milstein scheme}

We will consider a case study of our method for the Milstein scheme of a linear SDE, as well as the direct extension to the $\theta$-Milstein scheme, where, according to our knowledge, sharp estimates of the discrete Lyapunov exponents in the sense \eqref{SharpSense} have not been established yet. 
 
\subsection{Milstein scheme of linear SDE}

As a particular case, we obtain a result of the exponential stability versus the exponential blow-up at infinite time for the Milstein scheme of linear stochastic differential equations. For $\lambda,\sigma \in \mathbb{R}$, we consider the equation \begin{equation}
    dX(t) = \lambda X(t)dt + \sigma X(t)dB(t)  \quad \text{for } t>0, \qquad 
    X(0) = x_0 \ne 0, 
\label{SDE:Continuous0}
\end{equation} 
where $B$ is the standard scalar Brownian motion. The exponential stability of the zero equilibrium of \eqref{SDE:Continuous0} can be easily found in the literature, see e.g. in \cite{higham2007almost}, where by the It\^o formula we have 
\begin{gather*}
      \limsup_{t\to\infty} \dfrac{1}{t} \log \big[ \mathbb{E}|X(t)|^2\big]^{1/2} = \lambda + \frac{\sigma^2}{2} , 
\end{gather*} 
and 
\begin{gather*}
   \limsup_{t\to\infty} \dfrac{1}{t} \log |X(t)| = \lambda - \frac{\sigma^2}{2} .
\end{gather*} 
Consequently, the solution $X$ is mean-square exponentially stable with the exponent $\lambda+ \sigma^2/2$ or almost-surely exponentially stable with the exponent $\lambda- \sigma^2/2$ if 
 $$(\lambda,\sigma)\in \mathbb S^{\mathsf{ms}}:= \left\{(x,y) \bigg|\,x + \frac{y^2}{2}<0 \right\},$$ or respectively,  $$(\lambda,\sigma)\in \mathbb S^{\mathsf{as}}:= \left\{(x,y)\bigg|\,x- \frac{y^2}{2}<0 \right\}.$$
 In contrast, the solution is mean-square/almost-surely exponentially blowing up with the same exponents if $(\lambda,\sigma)\in \mathbb B^{\mathsf{ms}} := \mathbb{R}^2 \setminus \overline {\mathbb S^{\mathsf{ms}} } $, or respectively, $(\lambda,\sigma)\in \mathbb B^{\mathsf{as}} := \mathbb{R}^2 \setminus \overline {\mathbb S^{\mathsf{as}} } $.  In \cite[Section 3]{higham2007almost}, 
 the preservation of the above behaviours was proved to be retained for the Milstein scheme \eqref{SDE:Continuous0} with a small enough step size $\Delta t$.
 Specifically, for $(\lambda,\sigma)\in \mathbb S^{\mathsf{ms}}$ or $(\lambda,\sigma)\in\mathbb S^{\mathsf{as}}$ and for any $\kappa>0$, the Milstein scheme's solution $\{X_n\}$ to   \eqref{SDE:Continuous0} satisfies the upper bounds 
\begin{align}
      \limsup_{n\to\infty} \dfrac{1}{t_n} \log \big[ \mathbb{E}|X_n|^2\big]^{1/2} \le (1-\kappa) \left( \lambda + \frac{\sigma^2}{2} \right), 
\label{Intro1}
\end{align}  
and
\begin{align}
     \limsup_{n\to\infty} \dfrac{1}{t_n} \log |X_n| \stackrel{\text{a.s.}}{\le} (1-\kappa) \left( \lambda - \frac{\sigma^2}{2}  \right) .
\label{Intro2}
\end{align}  
However, the exponential blow-up, or more clearly, estimating the discrete Lyapunov exponent from below, has not yet been studied. 

\medskip
 
By letting $\varepsilon=0$, System \eqref{SDE:System} becomes an uncoupling one, where its equations exactly coincide with the equation \eqref{SDE:Continuous0}, and so does the calculation \eqref{For:LogXY}. Therefore, as a direct consequence of Theorem \ref{Theo:Main}, we obtain the following result.

\begin{proposition}
\label{Theo:System}
Let $\{X_n\}$ be the solution to the Milstein scheme \eqref{SDE:Continuous0} with a step size $0<\Delta t<1$. Then, there exist $0<\Delta t_2<1$ and $C_3,C_4>0$ such that the  estimates    
\begin{gather*}
    \lambda + \dfrac{\sigma^2}{2} - C_3\Delta t  \le \limsup_{n\to\infty} \dfrac{1}{t_n} \log \big[ \mathbb{E}|X_n|^2\big]^{1/2}  \le  \lambda + \dfrac{\sigma^2}{2} + C_3\Delta t, 
    \\ 
 \lambda - \dfrac{\sigma^2}{2} - C_4\Delta t^{1/2}   \stackrel{\text{a.s.}}{\le}  \limsup_{n\to\infty} \dfrac{1}{t_n} \log |X_n| \stackrel{\text{a.s.}}{\le}   \lambda - \dfrac{\sigma^2}{2} + C_4\Delta t^{1/2} 
\end{gather*}
hold for all $0<\Delta t< \Delta t_2$. Consequently, in this interval of step size, 
\begin{itemize}
    \item[a)] if 
 $(\lambda,\sigma)\in \mathbb S^{\mathsf{ms}}$ or $(\lambda,\sigma)\in \mathbb B^{\mathsf{ms}}$ then $X_n$ is  mean-square exponentially stable, or respectively,  mean-square exponentially blowing up with the exponent $\lambda+\frac{\sigma^2}{2}$; and 
    \item[b)] if 
 $(\lambda,\sigma)\in \mathbb S^{\mathsf{as}}$ or $(\lambda,\sigma)\in \mathbb B^{\mathsf{as}}$ then $X_n$ is almost-sure exponentially stable, or respectively,  almost-sure exponentially blowing up with the exponent $\lambda-\frac{\sigma^2}{2}$.
\end{itemize}
\end{proposition}

\subsection{The $\theta$ -Milstein scheme} 

Using the proposed method, this part is to obtain the preservation of the exponential stability and exponential blow-up behaviours for the theta-Milstein scheme of \eqref{SDE:Continuous0}, where the explicit term $ \lambda X_{n-1} \Delta t$ is replaced by  $[\lambda \theta X_{n} +\lambda (1-\theta)X_{n-1}]\Delta t$. Precisely, the theta-Milstein scheme is given, for any $0\le \theta \le 1$, by 
\begin{align}
\left\{\begin{aligned}
    &X^{\theta}_{n} = X^{\theta}_{n-1}+ [\lambda \theta X^{\theta}_{n} +\lambda (1-\theta)X_{n-1}]\Delta t +\sigma X^{\theta}_{n-1} \Delta B_{n-1} +\frac{\sigma^2}{2} X^{\theta}_{n-1} (\Delta B_{n-1}^2 -\Delta t), \\ &X^{\theta}_0 =x_0,
\end{aligned} \right.
\label{Theta:Milstein}
\end{align}
for $n\in \mathbb{N}$, $n\ge 1$, see e.g. \cite{Kloeden1992numerical}.  
By factorising the common factor $X^{\theta}_{n-1}$ on the right-hand side of \eqref{Theta:Milstein}, the induction shows that  
\begin{align}
    \begin{aligned}
        X^{\theta}_n  &= X^{\theta}_{n-1}  \left [  \frac{1+\lambda (1-\theta)\Delta t}{1-\lambda \theta \Delta t} + \frac{\sigma \Delta B_{n-1}}{1-\lambda \theta \Delta t} +\frac{\sigma^2}{2} \frac{\Delta B_{n-1}^2-\Delta t}{1-\lambda \theta \Delta t} \right]  \\
     &= x_0 \prod_{k=0}^{n-1} \left [ \eta_{\Delta t} + \frac{\sigma \Delta B_{k}+\frac{\sigma^2}{2}  \Delta B_{k}^2}{1-\lambda \theta \Delta t} \right] ,
    \end{aligned}
    \label{Theta:SolForm}
\end{align}
where we denote $$\eta_{\Delta t}:= \frac{1 + (\lambda(1-\theta)- \sigma^2/2) \Delta t}{1-\lambda \theta \Delta t}  .$$ Let us obtain likewise expressions as \eqref{Limit:1}, \eqref{Limit:2a} and \eqref{Limit:2b}. 
First, regarding the mean-square sense, it follows from \eqref{Theta:SolForm} that 
\begin{align*}
 (X^{\theta}_n)^2 = x_0^2 \prod_{k=0}^{n-1} \left [ \eta_{\Delta t}^2 + 2 \eta_{\Delta t} \dfrac{\sigma \Delta B_{k}+\frac{\sigma^2}{2}  \Delta B_{k}^2}{1-\lambda \theta \Delta t} + \left( \dfrac{\sigma \Delta B_{k}+\frac{\sigma^2}{2}  \Delta B_{k}^2}{1-\lambda \theta \Delta t} \right)^2 \right] . 
\end{align*}
 Thus, recalling the basic expectations in the proof of Proposition \ref{Theo:1}, we get 
\begin{align*}
\mathbb E |X^{\theta}_n|^2  
&= x_0^2 \prod_{k=0}^{n-1} \left [ \eta_{\Delta t}^2 +  \dfrac{\sigma^2 \eta_{\Delta t}  \Delta t}{1-\lambda \theta \Delta t} +  \dfrac{\sigma^2\Delta t + \frac{3\sigma^2}{4} \Delta t^2}{(1-\lambda \theta \Delta t)^2}  \right] \\
&= x_0^2 \left [ \eta_{\Delta t}^2 +  \dfrac{\sigma^2 \eta_{\Delta t}  \Delta t}{1-\lambda \theta \Delta t} +  \dfrac{\sigma^2\Delta t + \frac{3\sigma^2}{4} \Delta t^2}{(1-\lambda \theta \Delta t)^2}  \right]^n .
\end{align*}
By taking the limit superior as $n\to \infty$ and using the Taylor expansion of the function $x\mapsto \log (\eta_{\Delta t}^2 + x)$, we get
\begin{align}
\begin{aligned}
        \limsup_{n\to \infty} \frac{1}{t_n} \log \mathbb E |X^{\theta}_n|^2 &= \frac{1}{\Delta t} \log \left( \eta_{\Delta t}^2 +  \dfrac{\sigma^2 \eta_{\Delta t}  \Delta t}{1-\lambda \theta \Delta t} +  \dfrac{\sigma^2\Delta t + \frac{3\sigma^2}{4} \Delta t^2}{(1-\lambda \theta \Delta t)^2} \right) \\
    &=  \frac{2}{\Delta t} \log \eta_{\Delta t} + \frac{1}{\Delta t}  \sum_{m=1}^\infty \frac{(-1)^{m-1}}{m} \left( \dfrac{\sigma^2 \eta_{\Delta t}  \Delta t}{\eta_{\Delta t}^2(1-\lambda \theta \Delta t)} +  \dfrac{\sigma^2\Delta t + \frac{3\sigma^2}{4} \Delta t^2}{\eta_{\Delta t}^2 (1-\lambda \theta \Delta t)^2} \right)^m.
\end{aligned}  
\label{Theta:Lya1}
\end{align}
Second, regarding the almost-sure sense, due to the formulation  \eqref{Theta:SolForm} and the Taylor expansion of the function $x\mapsto \log |\eta_{\Delta t} + x|$, we have 
\begin{align}
     \begin{aligned}
         \limsup_{n\to \infty} \frac{1}{t_n} \log |X^{\theta}_n| 
       & = \limsup_{n\to \infty} \frac{1}{t_n}  \sum_{k=0}^{n-1} \log \left | \eta_{\Delta t} + \frac{\sigma \Delta B_{k}+\frac{\sigma^2}{2}  \Delta B_{k}^2}{1-\lambda \theta \Delta t} \right| \\
   & =  \frac{1}{\Delta t} \log \eta_{\Delta t} + \limsup_{n\to \infty} \frac{1}{t_n} \sum_{k=0}^{n-1} \sum_{j=1}^\infty     
\frac{(-1)^{j}}{j} \left( \frac{\sigma \Delta B_{k}+\frac{\sigma^2}{2}  \Delta B_{k}^2}{\eta_{\Delta t}(1-\lambda \theta \Delta t)} \right)^j  . 
     \end{aligned}
     \label{Theta:Lya2}
\end{align} 
For small enough $\Delta t$, by the same method as Propositions \ref{Theo:1}-\ref{Theo:2}, we directly obtain the following result.

\begin{proposition}
    [Theta-Milstein scheme] Let $0\le \theta \le 1$.
Let $\{X_n^\theta\}$ be the solution to the Milstein scheme \eqref{Theta:Milstein} with a step size $0<\Delta t<1$. Then, there exist $0<\Delta t_3<1$ and $C_5,C_6>0$ such that the estimates    
\begin{gather*}
    \lambda + \dfrac{\sigma^2}{2} - C_5\Delta t  \le \limsup_{n\to\infty} \dfrac{1}{t_n} \log \big[ \mathbb{E}|X_n^\theta|^2\big]^{1/2}  \le  \lambda + \dfrac{\sigma^2}{2} + C_5\Delta t, 
    \\ 
 \lambda - \dfrac{\sigma^2}{2} - C_6\Delta t^{1/2}   \stackrel{\text{a.s.}}{\le}  \limsup_{n\to\infty} \dfrac{1}{t_n} \log |X_n^\theta| \stackrel{\text{a.s.}}{\le}   \lambda - \dfrac{\sigma^2}{2} + C_6\Delta t^{1/2} 
\end{gather*}
hold for all $0<\Delta t< \Delta t_2$. Consequently, in this interval of step size, 
\begin{itemize}
    \item[a)] if 
 $(\lambda,\sigma)\in \mathbb S^{\mathsf{ms}}$ or $(\lambda,\sigma)\in \mathbb B^{\mathsf{ms}}$ then $X^{\theta}_n$ is  mean-square exponentially stable, or respectively,  mean-square exponentially blowing up with the exponent $\lambda+\frac{\sigma^2}{2}$; and 
    \item[b)] if 
 $(\lambda,\sigma)\in \mathbb S^{\mathsf{as}}$ or $(\lambda,\sigma)\in \mathbb B^{\mathsf{as}}$ then $X^{\theta}_n$ is almost-sure exponentially stable, or respectively,  almost-sure exponentially blowing up with the exponent $\lambda-\frac{\sigma^2}{2}$.
\end{itemize} 
\end{proposition}

\section{Further discussion and numerical simulation}
 
\subsection{Further discussion with different noisy intensities}

In this part, we discuss the case of the different noise intensities driven by the Brownian motion $B_2(t)$ on the solution components $X,Y$ of \eqref{SDE:System}. More specifically, we consider the system  
\begin{equation}
    d \begin{pmatrix}
    X \\
    Y 
    \end{pmatrix} \!=\! 
    \begin{pmatrix}
    \lambda & 0 \\
    0 & \lambda
    \end{pmatrix} 
    \begin{pmatrix}
    X \\
    Y 
    \end{pmatrix} dt 
    + 
    \begin{pmatrix}
    \sigma & 0 \\
    0 & \sigma
    \end{pmatrix}\!\! 
    \begin{pmatrix}
    X \\
    Y 
    \end{pmatrix}  dB_1(t)  +   \begin{pmatrix}
    0 & - \varepsilon_1 \\
    \varepsilon_2 & 0
    \end{pmatrix} \!\!
    \begin{pmatrix}
    X \\
    Y 
    \end{pmatrix}  dB_2(t), 
\label{SDE:Eps1Eps2}
\end{equation}
By the It\^o formula, 
\begin{align}
    d \log|Z(t)| = F(Z(t)) dt + \sigma dB_1(t) + G(Z(t)) dB_2(t) ,
    \label{SDE:Eps1Eps2:Log}
\end{align}
where, for $z=(x,y) \in \mathbb{R}^2$, the nonlinearities $F$, $G$ are given by  
\begin{align*}
    F(z) = \lambda -\dfrac{\sigma^2}{2} + \frac{1}{2|z|^4} \Big[  (\varepsilon_1 y^2+\varepsilon_2 z^2)^2+y^2 z^2(2\varepsilon_1 \varepsilon_2 -\varepsilon_1^2-\varepsilon_2^2)\Big]  
\end{align*}
and $G(z)=(\varepsilon_2-\varepsilon_1)xy/|z|^2$.   

\medskip

Compared to Equation \eqref{SDE:System}, the nonlinearities $F,G$ make it much more difficult to sharply estimate the Lyapunov exponent of the corresponding Milstein scheme. Although the above nonlinearities are obviously bounded, estimating the discrete Lyapunov exponent from below is challenging. This will be subject of future investigation.  

\subsection{Numerical simulations}

In this section, we present some numerical simulations to illustrate the obtained results, focusing on the exponential stability versus exponential blow-up of the solution $Z_n$ to the Milstein scheme \eqref{FinalLable} in the almost-sure sense. 

\medskip

By fixing $\Delta t=10^{-3}$, and take $t_n=n\Delta t=n10^{-3}$, we will numerically observe large-time behaviours of the discrete solution $Z_n$ as the discretised time $t_n$ increases up to $n=10^4$. We plot the logarithms $\log|Z_n|$ from fifty pathwise solutions by randomly generating the Brownian motion fifty times. 

\medskip

Let us first discuss the exponential blow-up of the discrete solution $Z_n$, i.e. $(\lambda,\varepsilon,\sigma)\in \mathbf B^{\mathsf{as}}$ due to Theorem \ref{Theo:Main}.  In Parts (a) and (b) of Figure \ref{fig:1}, we respectively take  $(\lambda,\varepsilon,\sigma)=(7,2,4)$ and $(\lambda,\varepsilon,\sigma)=(8,2,4)$, where the bold red lines are the averages of the logarithms $\log|Z_n|$ of the corresponding pathwise solutions, which behave as straight lines. We see that the slopes of these lines are positively small, and so are the discrete Lyapunov exponents, which also reflect small values of the continuum Lyapunov exponent $\lambda+\varepsilon^2/2-\sigma^2/2$.  Furthermore, in the case $(\lambda,\varepsilon,\sigma)=(7,2,4)$, the red line is closer to the horizontal axis since the discrete Lyapunov exponent is smaller, and many pathwise solutions may not be blowing up, or even, oscillating around the horizontal axis (since their logarithms oscillate around the horizontal axis).

\begin{figure}[H]
\begin{center}
\begin{subfigure}{.45\linewidth}
\centering
\includegraphics[width=\textwidth]{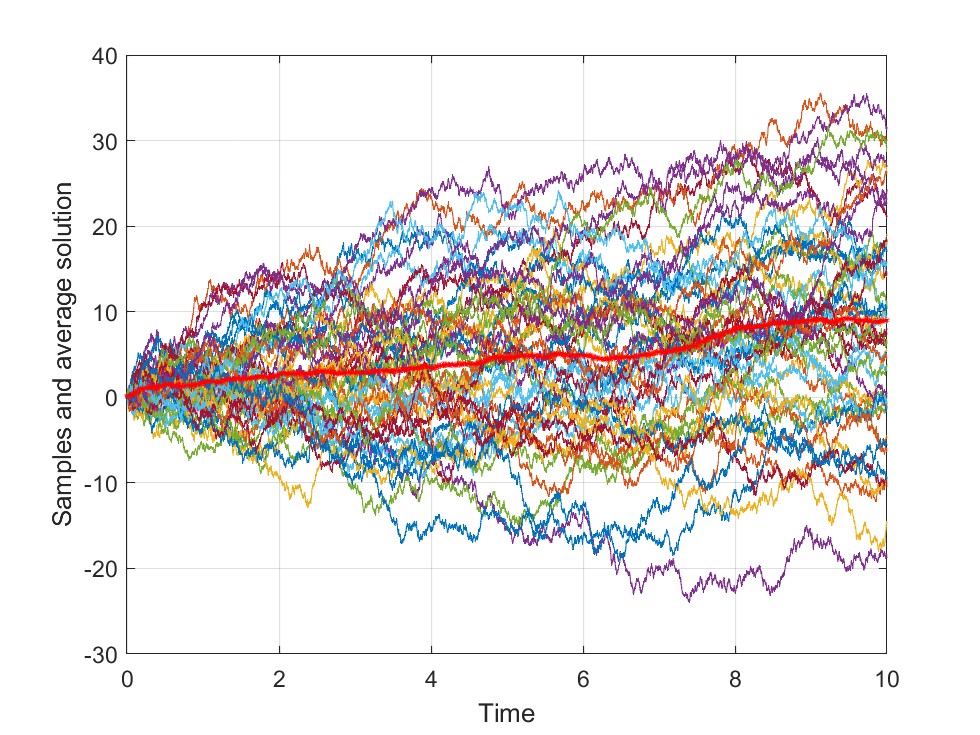}
\caption{ $(\lambda,\varepsilon,\sigma)=(7,2,4)$}
\label{fig:1a}
\end{subfigure}%
\vspace{0.2cm}
\begin{subfigure}{.45\linewidth}
\centering
\includegraphics[width=\textwidth]{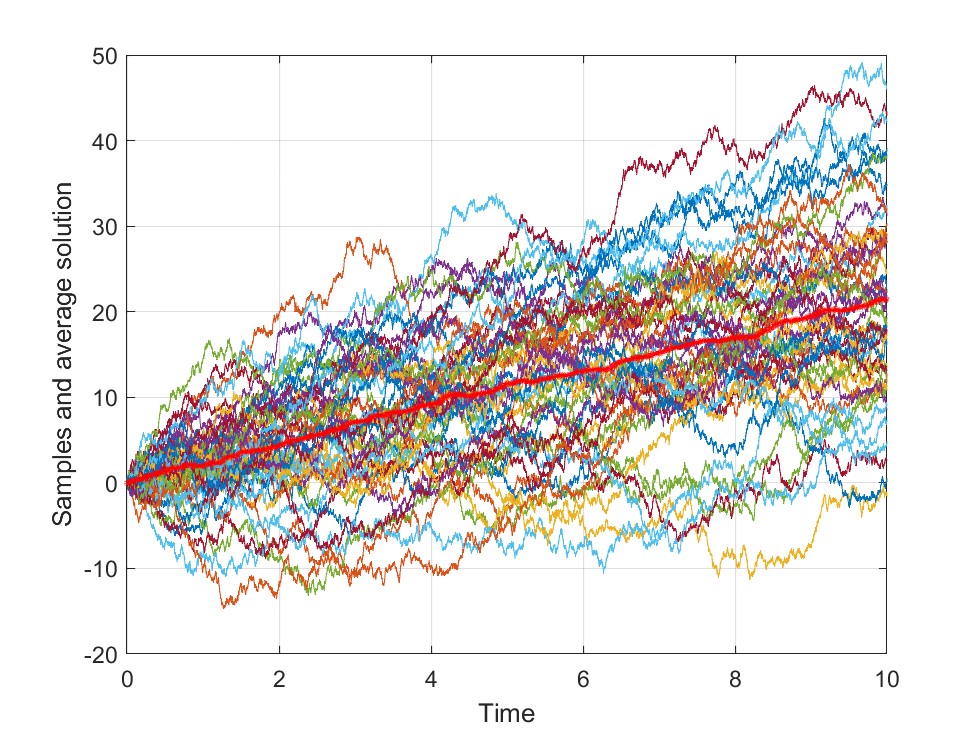}
\caption{ $(\lambda,\varepsilon,\sigma)=(8,2,4)$}
\label{fig:1b}
\end{subfigure}%

\begin{subfigure}{.45\linewidth}
\centering
\includegraphics[width=\textwidth]{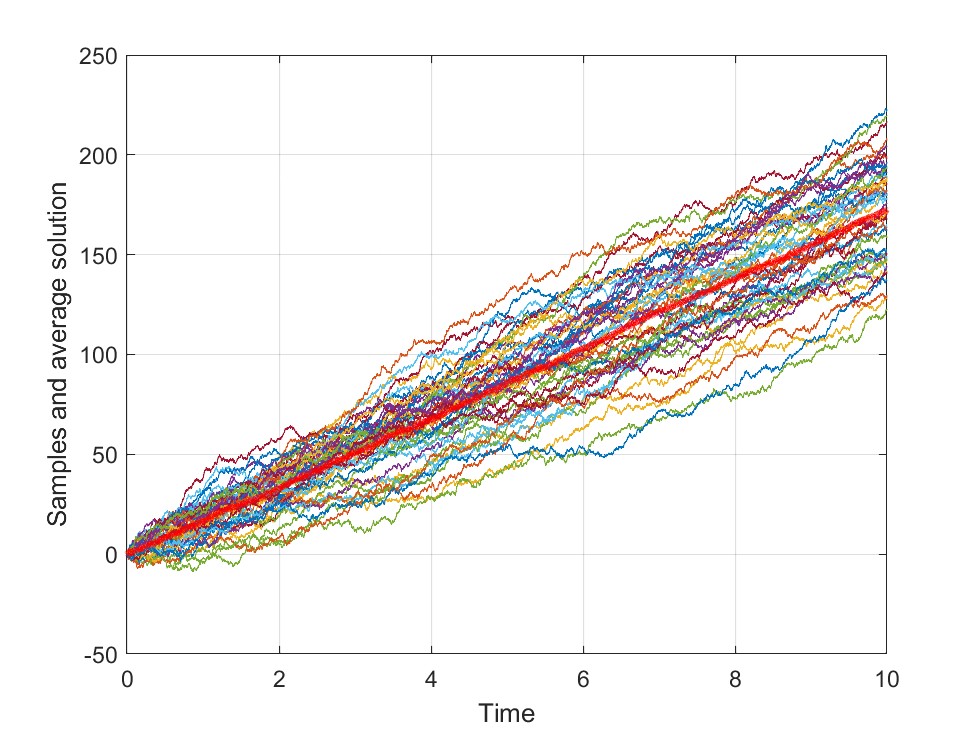}
\caption{ $(\lambda,\varepsilon,\sigma)=(30,6,8)$}
\label{fig:1c}
\end{subfigure}%
\begin{subfigure}{.45\linewidth}
\centering
\includegraphics[width=\textwidth]{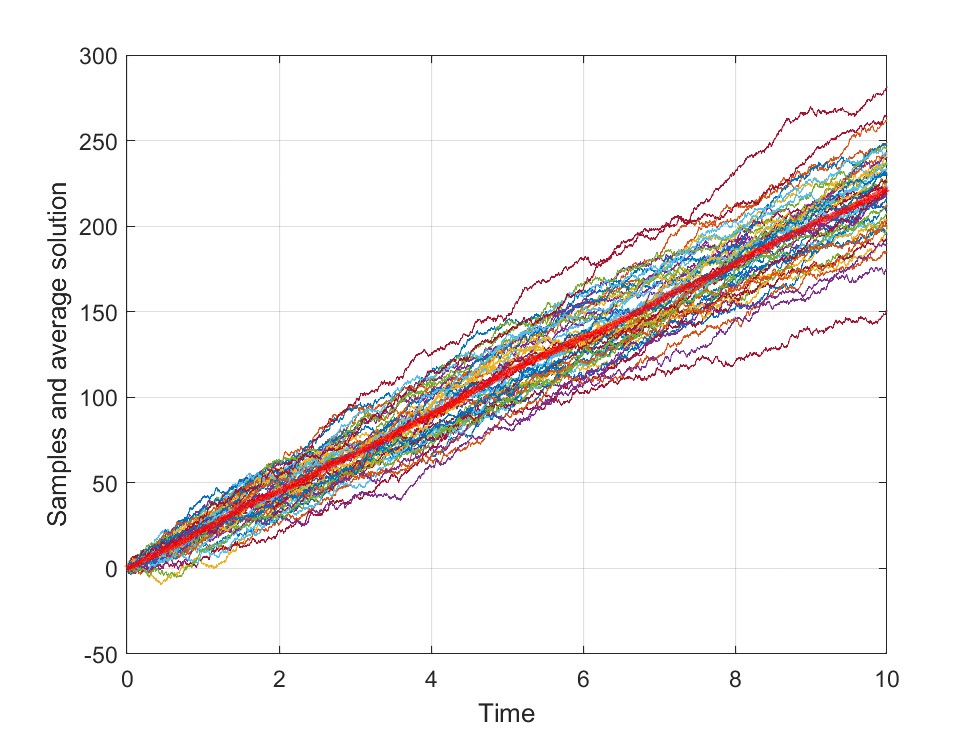}
\caption{$(\lambda,\varepsilon,\sigma)=(2,-10,8)$}
\label{fig:1d}
\end{subfigure}%
\caption{Almost-sure exponential blow-up of $| Z_n |$ via the graph of $\log|Z_n|$.}
\label{fig:1}
\end{center}
\end{figure} 

 Parts (c) and (d) of Figure \ref{fig:1} plot with  $(\lambda,\varepsilon,\sigma)=(30,6,8)$ and $(\lambda,\varepsilon,\sigma)=(2,-10,8)$, where the red lines have moderate slopes. The pathwise solutions oscillate with smaller amplitudes compared to Parts (a)-(b). Moreover, all the pathwise solutions' logarithms tend to increase after a long time.

 \medskip

In Figure \ref{fig:2}, the simulation illustrates fifty other pathwise solutions in the case that the continuum Lyapunov exponent is strictly negative, which shows the exponential stability of the zero equilibrium of the Milstein scheme. Similarly to the blowing-up behaviours in Figure \ref{fig:1}, we can observe likely properties of the sample, such as many pathwise solutions may not decay to zero or oscillate around the horizontal axis when the continuum Lyapunov exponent is negative and close to zero, as well as all the pathwise solution's logarithms tend to go down after a long time when the continuum Lyapunov exponent is more negative. 

\medskip

Now, let us vary the time step size $\Delta t$, from $10^{-1}$ to $10^{-5}$, to compare the discrete Lyapunov exponent with the continuum one, shown in Figure \ref{fig:3}, which demonstrate our sharp estimate in Theorem \ref{Theo:Main}. In other words, it says the equivalence of the exponential stability versus exponential blow-up between the Milstein scheme \eqref{FinalLable} with $0<\Delta t< \Delta t_1$ (for a small enough time step $\Delta t_1$) and the underlying model. Both subfigures show that these exponents are close to each other when the step size is of the order $O(10^{-3})$. Consequently, the mentioned equivalence can be observed when $\Delta t_1=O(10^{-3})$, which has not been obtained in our analysis.

\begin{figure}[H]
\begin{center}
\begin{subfigure}{.45\linewidth}
\centering
\includegraphics[width=\textwidth]{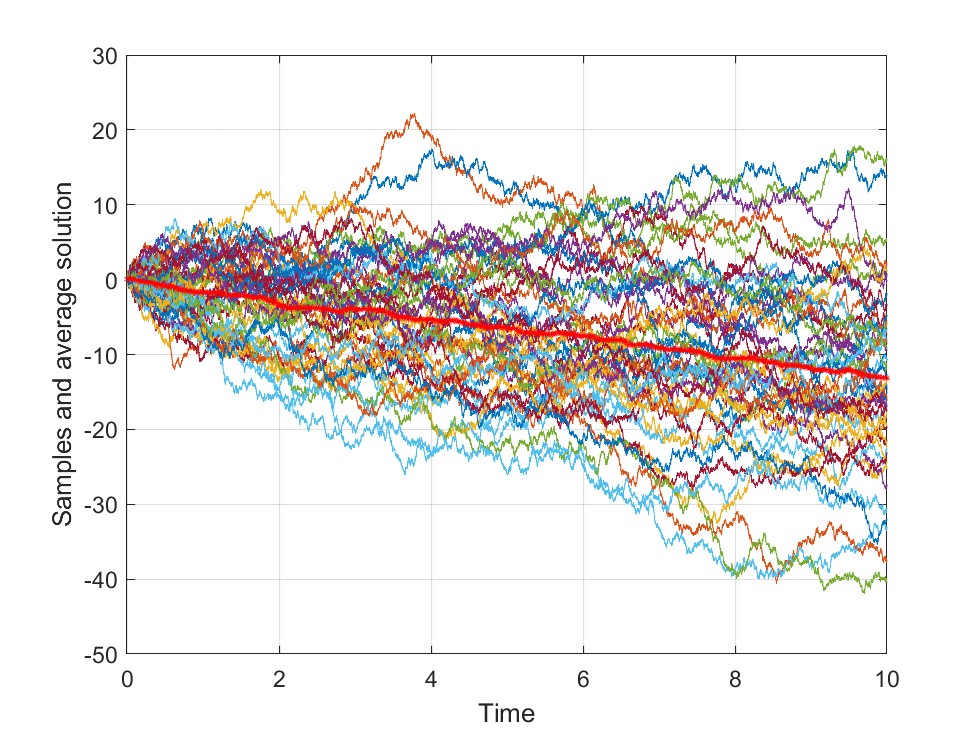}
\caption{$(\lambda,\varepsilon,\sigma)=(0.2,3.5,4)$}
\label{fig:2a}
\end{subfigure}%
\vspace{0.2cm}
\begin{subfigure}{.45\linewidth}
\centering
\includegraphics[width=\textwidth]{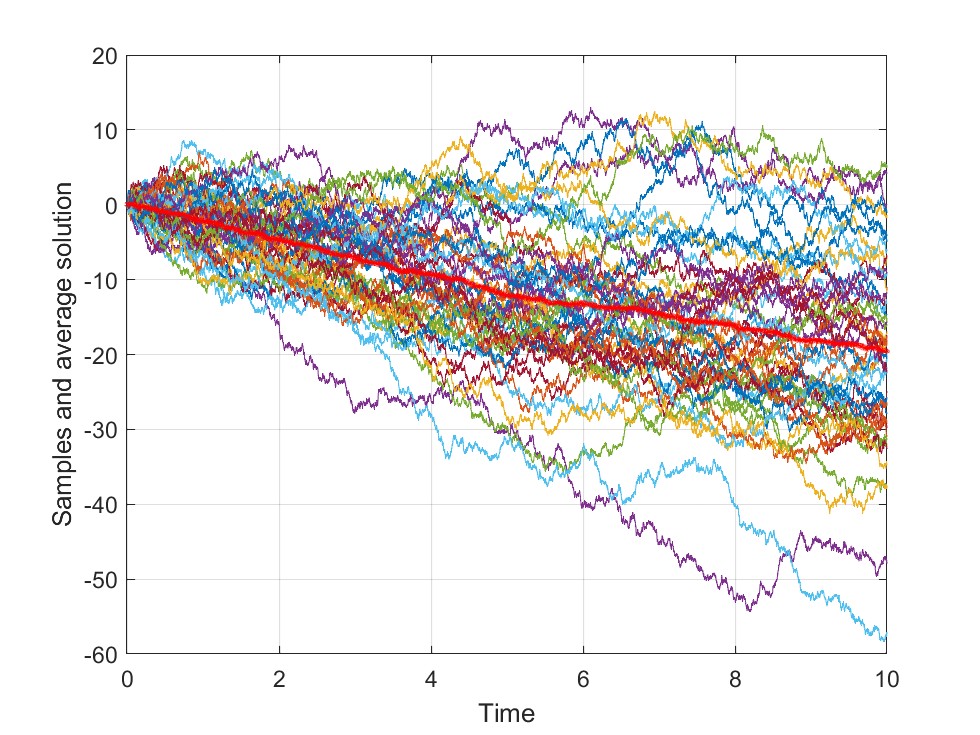}
\caption{$(\lambda,\varepsilon,\sigma)=(6,0.5,4)$}
\label{fig:2b}
\end{subfigure}%

\begin{subfigure}{.45\linewidth}
\centering
\includegraphics[width=\textwidth]{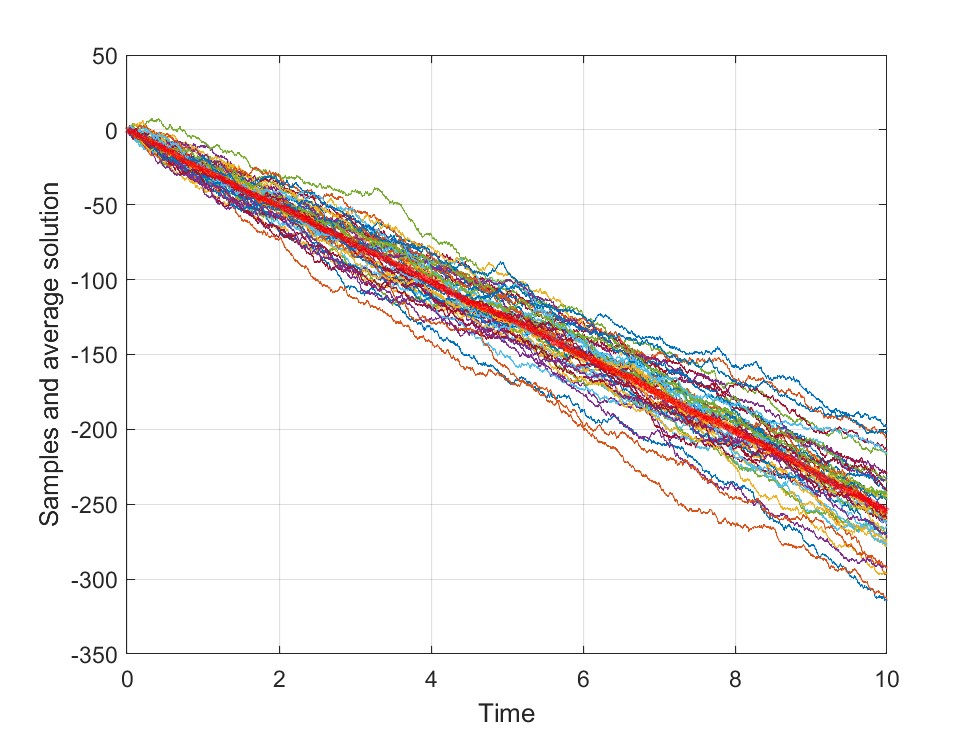}
\caption{$(\lambda,\varepsilon,\sigma)=(6,0.5,8)$}
\label{fig:2c}
\end{subfigure}%
\begin{subfigure}{.45\linewidth}
\centering
\includegraphics[width=\textwidth]{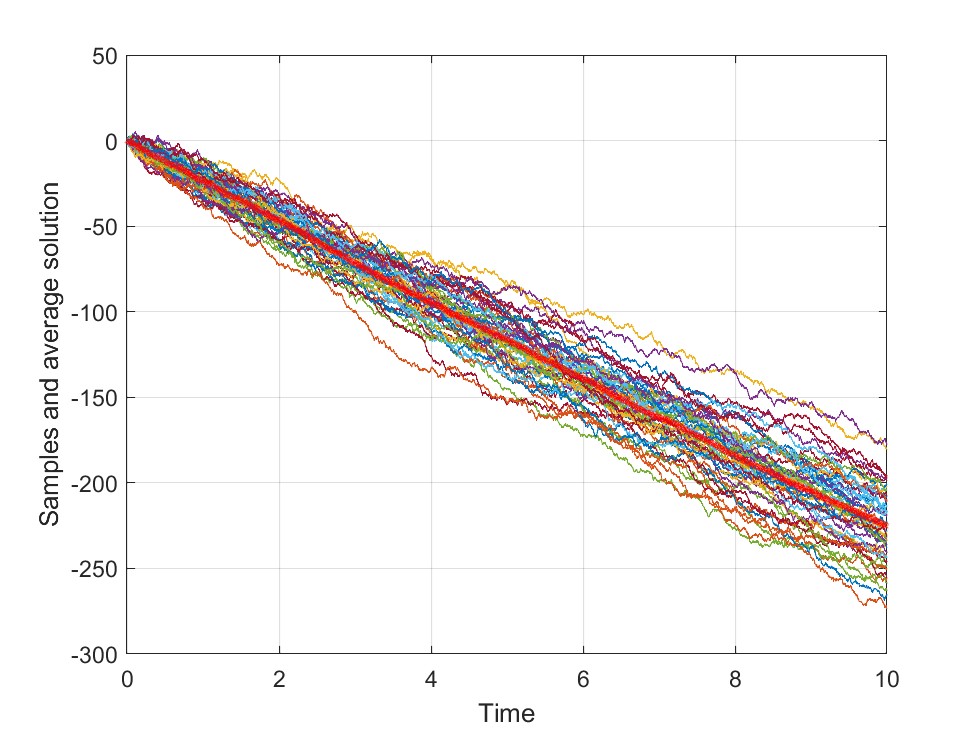}
\caption{$(\lambda,\varepsilon,\sigma)=(0.5,4,8)$}
\label{fig:2d}
\end{subfigure}%
\caption{Almost-sure exponential stability of $| Z_n |$ via the graph of $\log|Z_n|$.}
\label{fig:2}
\end{center}
\end{figure}


\begin{figure}[H]
\begin{center}
\begin{subfigure}{.45\linewidth}
\centering
\includegraphics[width=\textwidth]{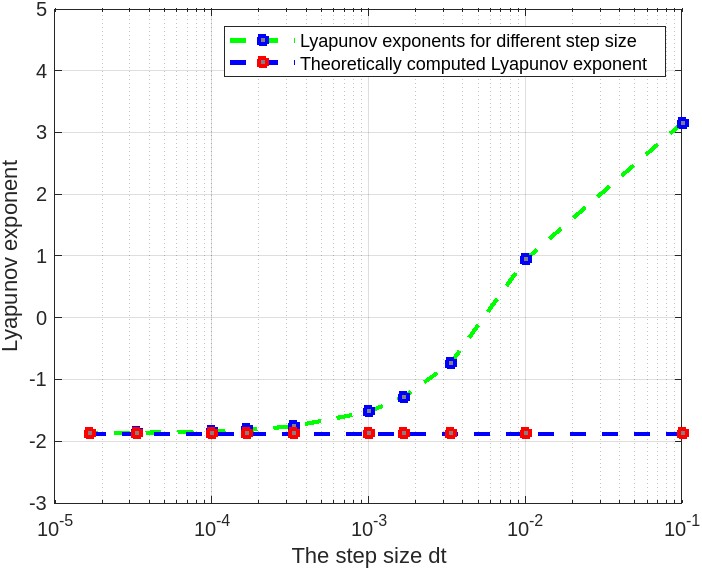}
\caption{ $(\lambda,\varepsilon,\sigma)=(6,0.5,4)$}
\label{fig:3a}
\end{subfigure}%
\vspace{0.2cm}
\begin{subfigure}{.45\linewidth}
\centering
\includegraphics[width=\textwidth]{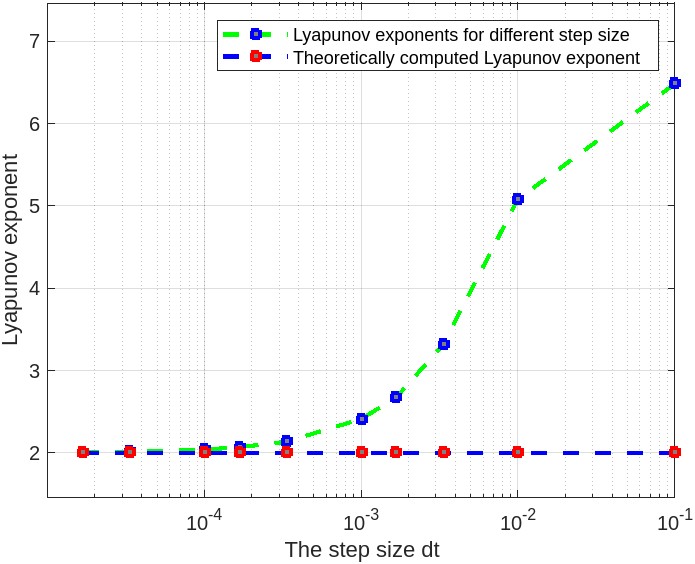}
\caption{ $(\lambda,\varepsilon,\sigma)=(8,2,4)$}
\label{fig:3b}
\end{subfigure}%
\caption{Continuum and discrete Lyapunov exponents as the time step $\Delta t$ varies.}
\label{fig:3}
\end{center}
\end{figure}
 
\appendix

\section*{Appendix}

\begin{proof}[Proof of Lemma \ref{Lem:BaIneqn}] The inequalities \eqref{BaIneqn1}-\eqref{BaIneqn2} are fundamental and can be directly proved by considering the monotonic properties of the following functions 
\begin{gather*}
    f(x):= \log(\gamma+x) - \log  \gamma  - \dfrac{1}{\gamma} x + \dfrac{1}{2\gamma^2}x^2-\dfrac{1}{3\gamma^3}x^3, \quad -\gamma < x < \infty, \\
    g(x):= \log(\gamma+x) - \log  \gamma  - \dfrac{1}{\gamma} x + \dfrac{1}{2\gamma^2}x^2 - \xi_\gamma(x),  \quad - \frac{2}{3}\gamma < x < \infty,
\end{gather*}
and their derivatives up to the fourth order, where splitting their domains into subsets of negative and nonnegative parts is needed. 

\medskip

We note that the inequality \eqref{BaIneqn1} is usual since it can be seen as a slight modification of \eqref{LogIneq}, we only present the proof of \eqref{BaIneqn2} here. The case $x=0$ is trivial. 
Let us first consider $x>0$, which corresponds to $\xi_\gamma(x)=-x^4/(4\gamma^4)$. Taking the derivatives gives
\begin{align*}  
    \left\{ \begin{array}{llrll}
   g'(x) = \dfrac{1}{\gamma +x} - \dfrac{1}{\gamma}  + \dfrac{1}{\gamma^2}x + \dfrac{1}{\gamma^4}x^3, \\
      g''(x) = \dfrac{-1}{(\gamma +x)^2}  +\dfrac{1}{\gamma^2} +\dfrac{3}{\gamma^4}x^2, \\
     g'''(x) = \dfrac{2}{(\gamma +x)^3} + \dfrac{6}{\gamma^4}x, \\
     g^{(4)}(x) = \dfrac{-6}{(\gamma +x)^4} + \dfrac{6}{\gamma^4}. 
    \end{array} \right.
 \end{align*}
 Since $ g^{(4)}(x)>0$ for all $x>0$, the third order derivative $g'''(x)$ is an  increasing function. Hence, $g'''(x)>g'''(0)>0$. This implies $g''(x)>g''(0)=0$ and, thus, $g'(x)>g'(0)=0$, i.e., $g$ is increasing on $(0,\infty)$, which ensures \eqref{BaIneqn2} for $x>0$. 

 \medskip
 
If $- \frac{2}{3}\gamma < x < 0$, we have $\xi_\gamma(x)=9x^3/\gamma^3$. Therefore,  
 \begin{align*}  
    \left\{ \begin{array}{llrll}
   g'(x) = \dfrac{1}{\gamma +x} - \dfrac{1}{\gamma}  + \dfrac{1}{\gamma^2}x - \dfrac{27}{\gamma^3}x^2, \\
      g''(x) = \dfrac{-1}{(\gamma +x)^2}  +\dfrac{1}{\gamma^2} -\dfrac{54}{\gamma^3}x, \\
     g'''(x) = \dfrac{2}{(\gamma +x)^3} - \dfrac{54}{\gamma^3}. 
    \end{array} \right.
 \end{align*}
In this interval of $x$, we have $g'''(x)<0$ and so $g''(x)>g''(0)=0$. This yields $g'(x)<g'(0)=0$, i.e., $g$ is decreasing on $(-2\gamma/3,0)$, which ensures \eqref{BaIneqn2} for $-2\gamma/3< x<0$. 
\end{proof}

\noindent {\bf \Large Acknowledgement} \; This research is funded by Hanoi University of Science and Technology (HUST) under project number T2024-PC-010. The author is grateful to Dr. Bao-Ngoc Tran for fruitful discussions, Ass.-Prof. Bao Quoc Tang for assistance concerning numerical simulations. Special thanks also go to Prof. Arturo Kohatsu-Higa, Assoc.-Prof. Hoang-Long Ngo and Dr. Ngoc Khue Tran for their helpful comments.

\vspace{0.3cm}

\noindent {\Large \bf Conflict of Interest}   The author declares that they have no conflict of interest.

\vspace{0.3cm}

\noindent {\Large \bf Data Availability}  Data sharing not applicable.

\vspace{0.3cm}
 

\begin{thebibliography}{CGAR03}

\bibitem[AMR08]{appleby2008stabilization}
John~AD Appleby, Xuerong Mao, and Alexandra Rodkina.
\newblock Stabilization and destabilization of nonlinear differential equations
  by noise.
\newblock {\em IEEE Transactions on Automatic Control}, 53(3):683--691, 2008.

\bibitem[BBKR12]{berkolaiko2012almost}
Gregory Berkolaiko, Evelyn Buckwar, C{\'o}nall Kelly, and Alexandra Rodkina.
\newblock Almost sure asymptotic stability analysis of the $\theta$-Maruyama
  method applied to a test system with stabilising and destabilising stochastic
  perturbations.
\newblock {\em LMS Journal of Computation and Mathematics}, 15:71--83, 2012.

\bibitem[BBKR13]{berkolaiko2013corrigendum}
Gregory Berkolaiko, Evelyn Buckwar, C{\'o}nall Kelly, and Alexandra Rodkina.
\newblock Corrigendum: On the use of a discrete form of the It{\^o} formula in
  the article ‘almost sure asymptotic stability analysis of the $\theta$-Maruyama
  method applied to a test system with stabilising and destabilising stochastic
  perturbations’.
\newblock {\em LMS Journal of Computation and Mathematics}, 16:366--372, 2013.

\bibitem[BBM85]{bellman1985stability}
Richard Bellman, Joseph Bentsman, and Semyon~M Meerkov.
\newblock Stability of fast periodic systems.
\newblock {\em IEEE Transactions on Automatic Control}, 30(3):289--291, 1985.

\bibitem[BK10]{buckwar2010towards}
Evelyn Buckwar and C{\'o}nall Kelly.
\newblock Towards a systematic linear stability analysis of numerical methods
  for systems of stochastic differential equations.
\newblock {\em SIAM Journal on Numerical Analysis}, 48(1):298--321, 2010.

\bibitem[BR20]{braverman2020global}
Elena Braverman and Alexandra Rodkina.
\newblock Global stabilization and destabilization by the state dependent noise
  with particular distributions.
\newblock {\em Physica D: Nonlinear Phenomena}, 403:132302, 2020.

\bibitem[CGAR03]{caraballo2003stochastic}
Tom{\'a}s Caraballo, Mar{\i}a~J Garrido-Atienza, and Jos{\'e} Real.
\newblock Stochastic stabilization of differential systems with general decay
  rate.
\newblock {\em Systems \& Control Letters}, 48(5):397--406, 2003.

\bibitem[Hig00]{higham2000mean}
Desmond~J Higham.
\newblock Mean-square and asymptotic stability of the stochastic Theta method.
\newblock {\em SIAM Journal on Numerical Analysis}, 38(3):753--769, 2000.

\bibitem[HMY04]{hu2004discrete}
Yaozhong Hu, Salah-Eldin~A Mohammed, and Feng Yan.
\newblock Discrete-time approximations of stochastic delay equations: the
  Milstein scheme.
\newblock {\em The Annals of Probability}, 32(1A):265--314, 2004.

\bibitem[HMY07]{higham2007almost}
Desmond~J Higham, Xuerong Mao, and Chenggui Yuan.
\newblock Almost sure and moment exponential stability in the numerical
  simulation of stochastic differential equations.
\newblock {\em SIAM Journal on Numerical Analysis}, 45(2):592--609, 2007.

\bibitem[KP92]{Kloeden1992numerical}
Peter~E. Kloeden and Eckhard Platen.
\newblock {\em Numerical solution of stochastic differential equations}.
\newblock Springer, 1992.

\bibitem[KPR13]{kelly2013almost}
C{\'o}nall Kelly, Peter Palmer, and Alexandra Rodkina.
\newblock Almost sure instability of the equilibrium solution of a
  Milstein-type stochastic difference equation.
\newblock {\em Computers \& Mathematics with Applications}, 66(11):2220--2230,
  2013.

\bibitem[Kus68]{kushner1968stability}
Harold~J Kushner.
\newblock On the stability of processes defined by stochastic
  difference-differential equations.
\newblock {\em Journal of Differential Equations}, 4(3):424--443, 1968.

\bibitem[Mao94]{mao1994stochastic}
Xuerong Mao.
\newblock Stochastic stabilization and destabilization.
\newblock {\em Systems \& Control Letters}, 23(4):279--290, 1994.

\bibitem[Mao15]{mao2015almost}
Xuerong Mao.
\newblock Almost sure exponential stability in the numerical simulation of
  stochastic differential equations.
\newblock {\em SIAM Journal on Numerical Analysis}, 53(1):370--389, 2015.

\bibitem[Mee82]{meerkov1982condition}
S~Meerkov.
\newblock Condition of vibrational stabilizability for a class of nonlinear
  systems.
\newblock {\em IEEE Transactions on Automatic Control}, 27(2):485--487, 1982.

\bibitem[MT04]{milstein2004stochastic}
Grigori~N Milstein and Michael~V Tretyakov.
\newblock {\em Stochastic numerics for mathematical physics}, volume~39.
\newblock Springer, 2004.

\bibitem[Shi96]{Shiryaev1996probability}
A.~N. Shiryaev.
\newblock {\em Probability}, volume 2nd ed.
\newblock Springer-Verlag, Berlin, 1996.

\bibitem[SM96]{saito1996stability}
Yoshihiro Saito and Taketomo Mitsui.
\newblock Stability analysis of numerical schemes for stochastic differential
  equations.
\newblock {\em SIAM Journal on Numerical Analysis}, 33(6):2254--2267, 1996.

\bibitem[WMS10]{wu2010almost}
Fuke Wu, Xuerong Mao, and Lukas Szpruch.
\newblock Almost sure exponential stability of numerical solutions for
  stochastic delay differential equations.
\newblock {\em Numerische Mathematik}, 115:681--697, 2010.

\end{thebibliography}

\end{document}